\documentclass[a4paper, 10pt]{article}
\usepackage[latin1]{inputenc} 
\usepackage[english]{babel}

\usepackage{amsmath}
\usepackage{amsthm}
\usepackage{amssymb}
\usepackage{graphicx}
\usepackage[ruled,linesnumbered]{algorithm2e}

\theoremstyle{plain}
\newtheorem{theorem}{Theorem}
\newtheorem{proposition}[theorem]{Proposition}
\newtheorem{lemma}[theorem]{Lemma}
\newtheorem{corollary}[theorem]{Corollary}

\newtheorem{assumption}[theorem]{Assumption}
\theoremstyle{remark}
\newtheorem{remark}[theorem]{Remark}

\newcommand{\R}{\mathbb{R}}
\newcommand{\dd}{\, \text{d}}

\usepackage{hyperref}
\usepackage{vmargin}
\setmarginsrb{3cm}{2.4cm}{3cm}{1.4cm}{0cm}{0cm}{0cm}{1.5cm}

\usepackage{subcaption}
\usepackage{caption}

\begin{document}

\title{Optimality Conditions in Variational Form for Non-Linear Constrained Stochastic Control Problems}
\author{Laurent Pfeiffer\footnote{Institute of Mathematics, University of Graz, Austria. E-mail: laurent.pfeiffer@uni-graz.at}}
\date{\today}

\maketitle

\begin{abstract}
Optimality conditions in the form of a variational inequality are proved for a class of constrained optimal control problems of stochastic differential equations. The cost function and the inequality constraints are functions of the probability distribution of the state variable at the final time.
The analysis uses in an essential manner a convexity property of the set of reachable probability distributions.
An augmented Lagrangian method based on the obtained optimality conditions is proposed and analyzed for solving iteratively the problem. At each iteration of the method, a standard stochastic optimal control problem is solved by dynamic programming. Two academical examples are investigated.
\end{abstract}



\section{Introduction}

This article is devoted to the derivation of optimality conditions for a class of control problems of stochastic differential equations (SDE). For a given adapted control process $u$, let $(X_t^{0,Y_0,u})_{t \in [0,T]}$ be the solution to the following SDE:
\begin{equation*}
\begin{cases}
\begin{array}{rl}
\dd X_t^{0,Y_0,u}= & b(X_t^{0,Y_0,u},u_t) \dd t + \sigma(X_t^{0,Y_0,u},u_t) \dd W_t, \quad \text{for a.\,e.\@ $t \in [0,T]$,} \\
\quad X_0^{0,Y_0,u}= & Y_0,
\end{array}
\end{cases}
\end{equation*}
where the drift $b$, the volatility $\sigma$, and the initial condition $Y_0$ are given. For all $t \in [0,T]$, we denote by $m_t^{0,Y_0,u}$ the probability distribution of $X_t^{0,Y_0,u}$.
Along the article, constrained problems of the following form are considered:
\begin{equation} \label{eqAbstractPb}
\inf_{u \in \mathcal{U}_0(Y_0)} F(m_T^{0,Y_0,u}), \quad \text{subject to: } G(m_T^{0,Y_0,u}) \leq 0,
\end{equation}
where the mappings $F$ and $G$ are given and satisfy differentiability assumptions. The set $\mathcal{U}_0(Y_0)$ is a set of adapted stochastic processes. A precise description of problem \eqref{eqAbstractPb} will be given in Section \ref{sectionFormulation}.
In this paper, we call the mappings $F$ and $G$ linear if they can be written in the form $F(m)= \int_{\R^n} f(x) \dd m(x)$ and $G(m)= \int_{\R^n} g(x) \dd m(x)$, where $f:\R^n \rightarrow \R$ and $g:\R^n \rightarrow \R^N$. In this specific case, problem \eqref{eqAbstractPb} is equivalent to the following stochastic optimal control problem with an expectation constraint:
\begin{equation} \label{eqPbExpCont}
\inf_{u \in \mathcal{U}_0(Y_0)} \mathbb{E} \big[ f \big(X_T^{0,Y_0,u} \big) \big], \quad
\text{subject to: } \mathbb{E} \big[ g \big( X_T^{0,Y_0,u} \big) \big] \leq 0.
\end{equation}
The terminology \emph{non-linear} used in the title refers to the fact that the functions $F$ and $G$ for which our result applies are not necessarily linear. In other words, the cost function and the constraints are not necessarily formulated as expectations of functions of $X_T^{0,Y_0,u}$. This is the specificity of the present article.

Stochastic optimal control problems with a non-linear cost function are mainly motivated by applications in economy and finance. In many situations, minimizing the expectation of a random cost may be unsatisfactory and one may prefer to take into account the ``risk" induced by the dispersion of the cost. In the literature, there exist many different models of risk and some of them can be formulated as functions of the probability distribution of the state variable, as explained for example in \cite[Chapter 6]{SDR14}. Some portfolio problems with risk-averse cost functions are studied in \cite{Pfe16}. In \cite[Section 5]{AD11}, a mean-variance portfolio selection problem is considered as well as in \cite{PG17}. In \cite{MY17}, a gradient-based method is developed for minimization problems of the conditional value at risk, a popular risk-averse cost function.
The risk can also be taken into account by considering a constraint of the form $G(m_T^{0,Y_0,u}) \leq 0$. For example, one can try to keep the probability of bankruptcy under a given threshold with a probability constraint. If $G$ models the variance of the outcome of some industrial process, then a constraint on the variance can guarantee some uniformity in the outcome, which can be a desirable property. Final-time constraints have several applications in finance, see for example \cite{Tou00}, where probability constraints are used for solving a super-replication problem or \cite{TT13}, where the final probability distribution is fixed for solving a determination problem of no-arbitrage bounds of exotic options.
Finally, let us mention that a problem of the form \eqref{eqAbstractPb} can be seen as a simple model for an optimization problem of a multi-agent system, where a government, that is to say a centralized controller, influences a (very large) population of agents, whose behavior is described by a SDE (see for example \cite{ACFK16} and the references therein on this topic). Let us mention however that for general multi-agent models, the coefficients of the SDE at time $t$ depend on the current probability distribution $m_t^{0,Y_0,u}$.

Let us describe the available results in the literature related to optimality conditions for problems similar to Problem \eqref{eqStandardPb}.
Problem \eqref{eqStandardPb}, without constraints, is a specific case of an optimal control problem of a McKean-Vlasov process (or mean-field-type control problem). For this more general class of problems, the drift and volatility of the SDE possibly depend at any time $t$ on the current distribution $m_t^{0,Y_0,u}$. Optimality conditions usually take the form of a stochastic maximum principle: for a solution $\bar{u}$, $\bar{u}_t$ minimizes almost surely and for almost every time $t$ a Hamiltonian involving a costate which is obtained as a solution to a backward stochastic differential equation (see for example \cite{AD11}, \cite{BFY13}, \cite{BDL11}, \cite{CD15}).
In a different but related approach, one can consider control processes in a feedback form, that is to say in the form $u_t= \mathbf{u}(t,X_t^{0,Y_0,u})$, where the mapping $\mathbf{u} \colon [0,T] \times \R^n \rightarrow U$ has to be optimized. Under regularity assumptions, the probability distribution $m_t^{0,Y_0,u}$ has a density, say $\mu(t,x)$, which is the solution to the Fokker-Planck equation:
\begin{equation*}
\partial_t \mu= \frac{1}{2} \nabla^2 :\big[ \mu(t,x) a \big( x,\mathbf{u}(t,x) \big) \big] - \nabla \cdot  \big[ \mu(t,x) b \big( x,\mathbf{u}(t,x) \big)  \big],
\end{equation*}
where $a= \sigma \sigma^\top$ and where the operators $\nabla \cdot$ and $\nabla^2 :$ are defined by
\begin{equation*}
\nabla \cdot f(x)= \sum_{i=1}^n \partial_{x_i}f_i(x) \quad
\text{and} \quad \nabla^2 : g(x)= \sum_{i,j=1}^n \partial_{x_ix_j}^2 g_{ij}(x),
\end{equation*}
respectively. A derivation of the Fokker-Planck can be found in \cite[Lemma 3.3]{Car12}, for example.
Note that if $b$ and $\sigma$ do not depend on $\mu$, then the Fokker-Planck equation is a linear partial differential equation.
In this approach, the problem is an optimal control problem of the Fokker-Planck equation and optimality conditions take the form of a Pontryagin's maximum principle. The adjoint equation is in this setting an HJB equation. We refer the reader to \cite{AL15}, \cite{ACFK16}, \cite{BFY13}, and \cite{FG17} for this approach.

In this article, optimal control problems of the following form:
\begin{equation} \label{eqStandardPb}
\inf_{u \in U_0(Y_0)} \mathbb{E}\big[ \phi(X_T^{0,Y_0,u}) \big]
\end{equation}
are called \emph{standard problems}, considering the fact that they have been extensively studied in the last decades. They can be solved by dynamic programming, by computing the solution to the associated Hamilton-Jacobi-Bellman equation (see the textbooks \cite{FS06}, \cite{Pha09}, \cite{YZ99} on this field).
Since standard problems are of form \eqref{eqPbExpCont} (without constraints),
they fall into the general class of problems investigated in the paper.
The optimality conditions provided in the present article can be shortly formulated as follows: if $\bar{u}$ is a solution to \eqref{eqAbstractPb} and satisfies a qualification condition, then it is also the solution to a standard problem of the form \eqref{eqStandardPb}, where the involved function $\phi$ is the derivative at $m_T^{0,Y_0,u}$ (in a specific sense) of the Lagrangian of the problem $L:= F + \langle \lambda, G \rangle$, for some non-negative Lagrange multiplier $\lambda$ satisfying a complementarity condition. Our optimality conditions therefore take the form of a variational inequality. Our analysis relies on the following technical result: the closure (for the Wasserstein distance associated with the $L^1$-distance) of the set of reachable probability distributions at time $T$ is convex. This property is proved by constructing controls imitating the behaviour of relaxed controls.
To the best of our knowledge, the optimality conditions for problem \eqref{eqPb} are new, as well as the convexity property satisfied by the reachable set\footnote{A proof of the convexity property as well as optimality conditions in variational form for unconstrained problems can be found in the unpublished research report \cite{Pfe15a}.}. They differ from the maximum principle mentioned above and are, to a certain extent, related to the optimality conditions obtained for mean-field-type control problems formulated with feedback laws. The presence of non-linear constraints is another novelty of the article; in the literature, most constraints of the form $G(m_T^{0,Y_0,u}) \leq 0$ are expectation constraints, see for example \cite{BET09}, \cite{Pfe15}. The existence of a Lagrange multiplier, even in a linear setting, is not often considered, see \cite[Section 5]{BS12} or \cite[Section 6]{P90}.

It is well-known that problems of the form \eqref{eqAbstractPb} are time-inconsistent, i.e.\@ it is (in general) not possible to write a dynamic programming principle by parameterizing problem \eqref{eqAbstractPb} by its initial time and initial condition, as is customary for standard problems of the form \eqref{eqStandardPb}. A dynamic programming principle can be written if one considers the whole initial probability distribution as a state variable, see  \cite{LP14}, \cite{PW16}, \cite{PW15}. However, in practice, this approach does not allow, in general, to solve the problem, because the complexity of the method grows exponentially with the dimension of the (discretized) space of probability distributions.
The optimality conditions in variational form and the convexity property proved in this article naturally lead to iterative methods for solving problem \eqref{eqAbstractPb}, based on successive resolutions of standard problems and thus overcoming the difficulty related to time-inconsistency. We propose, analyse, and test such a method in the article. The cost function of the standard problem to be solved at each iteration is the derivative, in a certain sense, of an augmented Lagrangian.

We give a precise formulation of the problem under study in Section \ref{sectionFormulation}. We also discuss the notion of differentiability which is used.
In Section \ref{sectionConvexity}, we prove the convexity of the closure of the reachable set of probability distributions.
Optimality conditions in variational form are proved in Section \ref{sectionOptimality}. The case of convex problems is discussed.
Our numerical method for solving the problem is described and analyzed in Section \ref{sectionNumerics}.
We provide results for two academical examples.
Elements on optimal transportation theory are given in the appendix.

\section{Formulation of the problem and assumptions} \label{sectionFormulation}

\subsection{Notation}

\begin{itemize}
\item The set of probability measures on $\R^n$ is denoted by $\mathcal{P}(\R^n)$. For a function $\phi:\R^n \rightarrow \R$, its integral (if well-defined) with respect to the measure $m \in \mathcal{P}(\R^n)$ is denoted by
\begin{equation*}
\int_{\R^n} \phi(x) \dd m(x) \quad \text{or} \quad
\int_{\R^n} \phi \dd m.
\end{equation*}
Given two measures $m_1$ and $m_2 \in \mathcal{P}(\R^n)$, we denote:
\begin{equation*}
\int_{\R^n} \phi(x) \dd (m_2(x)- m_1(x)) := \int_{\R^n} \! \phi(x) \dd m_2(x) - \int_{\R^n} \! \phi(x) \dd m_1(x).
\end{equation*}
\item For a given random variable $X$ with values in $\R^n$, its probability distribution is denoted by $\mathcal{L}(X) \in \mathcal{P}(\R^n)$. If $m= \mathcal{L}(X) \in \mathcal{P}(\R^n)$, then for any continuous and bounded function $\phi:\R^n \rightarrow \R$,
\begin{equation*}
\mathbb{E}\big[ \phi(X) \big]= \int_{\R^n} \phi \dd m.
\end{equation*}
We also denote by $\sigma(X)$ the $\sigma$-algebra generated by $X$.
\item For a given vector $x \in \R^q$, we denote by $| x |$ its Euclidean norm and by $| x |_\infty$ its supremum norm.
\item For $p \geq 1$, we denote by $\mathcal{P}_p(\R^n)$ the set of probability measures having a finite $p$-th moment:
\begin{equation*}
\mathcal{P}_p(\R^n) := \Big\{
m \in \mathcal{P}(\R^n) \,\big|\, \int_{\R^n} | x |^p \dd m(x) < + \infty
\Big\}.
\end{equation*}
We recall that for $1 \leq p \leq q$, the space $\mathcal{P}_q(\R^n)$ is included into $\mathcal{P}_p(\R^n)$.
We equip $\mathcal{P}_1(\R^n)$ with the Wasserstein distance $d_1$ (the definition is given in the appendix). We recall the dual representation of $d_1$ \cite[Remark 6.5]{Vil09}: for all $m_1$, $m_2 \in \mathcal{P}_1(\R^n)$,
\begin{equation} \label{eqDualWasserstein}
d_1(m_1,m_2)= \sup_{\phi \in \text{1-Lip}(\R^n)} \int_{\R^n} \phi \dd(m_2- m_1),
\end{equation}
where $\text{1-Lip}(\R^n)$ is the set of real-valued Lipschitz continuous functions of modulus 1.
\item For all $R\geq 0$, we define
\begin{equation} \label{eqDefBpR}
\bar{B}_p(R):= \Big\{ m \in \mathcal{P}_p(\R^n) \,|\, \int_{\R^n} | x |^p \dd m(x) \leq R \Big\}.
\end{equation}
\item The open (resp.\@ closed) ball in $\R^n$ of radius $r \geq 0$ and center $0$ is denoted by $B_r$ (resp.\@ $\bar{B}_r$), its complement $B_r^\text{c}$ (resp.\@ $\bar{B}_r^\text{c}$), for the Euclidean norm.
\item For a given $p \geq 1$, we say that a function $\phi:\R^n \rightarrow \R^q$ is dominated by $|x|^p$ if for all $\varepsilon>0$, there exists $r>0$ such that for all $x \in B_r^{\text{c}}$,
\begin{equation} \label{eqContPropR}
| \phi(x) | \leq \varepsilon | x |^p.
\end{equation}
\item The convex envelope of a set $\mathcal{R}$ is denoted $\text{conv}(\mathcal{R})$. When $\mathcal{R}$ is a subset of $\mathcal{P}_1(\R^n)$, its closure for the $d_1$-distance is denoted $\text{cl}(\mathcal{R})$.
\end{itemize}

\subsection{State equation}

We fix a final time $T>0$ and a Brownian motion $(W_t)_{t \in [0,T]}$ of dimension $d$.
For all $0 \leq t \leq T$, $(W_s-W_t)_{s \in [t,T]}$ is a standard Brownian motion. For all $s \in [t,T]$, we denote  by $\mathcal{F}_{t,s}$ the $\sigma$-algebra generated by $(W_\theta-W_t)_{\theta \in [t,s]}$. 

Let $U$ be a compact subset of $\R^k$.  Note that we do not make any other assumption on $U$: it can be non-convex, for example, or can be a discrete set. For a given random variable $Y_t$ independent of $\mathcal{F}_{t,T}$ with values in $\R^n$, we define the sets $\mathcal{U}_t^0$ and $\mathcal{U}_t(Y_t)$ as the sets of control processes $(u_s)_{s \in [t,T]}$ taking values in $U$ such that for all $s \in [t,T]$, $u_s$ is respectively $\mathcal{F}_{t,s}$-measurable and $(\sigma(Y_t)\times \mathcal{F}_{t,s})$-measurable, respectively.

The drift $b:\R^n \times U \rightarrow \R^n$ and the volatility $\sigma:\R^n \times U \rightarrow \R^{n \times d}$ are given. For all $u \in \mathcal{U}_t(Y_t)$, we denote by $\big( X_s^{t,Y_t,u} \big)_{s \in [t,T]}$ the solution to the SDE
\begin{equation} \label{eqGenSDE}
\dd X_s^{t,Y_t,u}= b(X_s^{t,Y_t,u},u_s) \dd s + \sigma(X_s^{t,Y_t,u},u_s) \dd W_s, \quad \forall s \in [t,T], \quad X_t^{t,Y_t,u}= Y_t.
\end{equation}
The well-posedness of this SDE is ensured by Assumption \ref{hypLipschitzCoeff} \cite[Section 5]{Oks03} below. We also denote by $m_s^{t,Y_t,u}$ the probability distribution of $X_s^{t,Y_t,u}$:
\begin{equation*}
m_s^{t,Y_t,u}= \mathcal{L}(X_s^{t,Y_t,u}).
\end{equation*}

All along the article, we assume that the following assumption holds true. From now on, the initial condition $Y_0$ and the real number $p \geq 2$ introduced below are fixed.

\begin{assumption} \label{hypLipschitzCoeff} There exists $K>0$ such that for all $x,y \in \R^n$, for all $u, v \in U$,
\begin{align*}
& | b(x,u) | + | \sigma(x,u) | \leq K(1+ | x | + | u |), \\
& | b(x,u) - b(y,v) | + | \sigma(x,u)-\sigma(y,v) |  \leq K( | y-x | + | v-u |).
\end{align*}
There exists $p \geq 2$ such that $\mathcal{L}(Y_0) \in \mathcal{P}_p(\R^n)$.
\end{assumption}

The following lemma is classical, see for example \cite[Section 2.5]{Kry80}.

\begin{lemma} \label{lemmaLipschitzEstimates}
There exist three constants $C_1$, $C_2$, and $C_3>0$ depending on $K$, $T$, $U$, and $p$ such that for all $t \in [0,T]$, for all random variables $Y_t$ and $\tilde{Y}_t$ independent of $\mathcal{F}_{t,T}$ and taking values in $\R^n$, for all $u \in \mathcal{U}_t(Y_t)$, for all $0 \leq h \leq T-t$, for all $t \leq s \leq T-h$,
the following estimates hold:
\begin{enumerate}
\item $\ \mathbb{E} \Big[ \sup_{t \leq \theta \leq T} \big| X_\theta^{t,Y_t,u} \big|^p \Big] \leq C_1 \big( 1+ \mathbb{E}\big[ |Y_t|^p \big] \big)$
\item $\ \mathbb{E} \Big[ \sup_{s \leq \theta \leq s+h} |X_\theta^{t,Y_t,u} - X_s^{t,Y_t,u}|^p \Big] \leq C_2 h^{p/2} \big(1+ \mathbb{E} \big[ |Y_t|^p \big] \big)$
\item $\ \mathbb{E} \Big[ \sup_{t \leq \theta \leq T} \big| X_\theta^{t,\tilde{Y}_t,u} - X_\theta^{t,Y_t,u} \big|^p \Big] \leq C_3 \mathbb{E} \big[ |\tilde{Y}_t-Y_t|^p \big]$.
\end{enumerate}
\end{lemma}

We denote by $\mathcal{R}$ the set of reachable probability distributions at time $T$, defined by
\begin{equation} \label{eqDefR}
\mathcal{R}= \{ m_T^{0,Y_0,u} \,|\, u \in \mathcal{U}_0(Y_0) \}.
\end{equation}
By Lemma \ref{lemmaLipschitzEstimates}, there exists $R>0$ such that
\begin{equation} \label{eqRisBounded}
\mathcal{R} \subseteq \bar{B}_p(R).
\end{equation}
By Lemma \ref{lemmaCompactnessProperty} (in the appendix), $\bar{B}_p(R)$ is compact for the $d_1$-distance, thus it is bounded. It follows that $\mathcal{R}$ and $\text{cl}(\mathcal{R})$ are bounded. We can therefore consider, for future reference, the diameter of $\text{cl}(\mathcal{R})$, defined by
\begin{equation} \label{eqDiameter}
D:= \sup_{m_1,m_2 \in \text{cl}(\mathcal{R})} d_1(m_1,m_2).
\end{equation}

\subsection{Formulation of the problem and regularity assumptions}

Let $F \colon \mathcal{P}_p(\R^n) \rightarrow \R$ and $G \colon \mathcal{P}_p(\R^n) \rightarrow \R^N$ be two given mappings.
We aim at studying the following problem:
\begin{equation*} \tag{$P$} \label{eqPb}
\inf_{u \in \mathcal{U}_0(Y_0)} \ F(m^{0,Y_0,u}_T) \quad \text{subject to: } G(m^{0,Y_0,u}_T) \leq 0.
\end{equation*}
Throughout the article, we assume that the next two assumptions, dealing with the continuity and the differentiability of $F$ and $G$, are satisfied. The constant $R$ used in these assumptions is given by \eqref{eqRisBounded}.

\begin{assumption} \label{hypContinuity}
The restrictions of $F$ and $G$ to $\bar{B}_p(R)$ are continuous for the $d_1$-distance.
\end{assumption}

In order to state optimality conditions, we need a notion of derivative for the mappings $F$ and $G$. Denoting by $\mathcal{M}(\R^n)$ the set of finite signed measures on $\R^n$, we define:
\begin{equation*}
\widehat{\mathcal{M}}_{p}(\R^n)
= \Big\{
m \in \mathcal{M}(\R^n) \,\big|\, \int_{\R^n} |x|^{p} \dd |m|(x) < + \infty, \ \int_{\R^n} 1 \dd m(x)= 0
\Big\}.
\end{equation*}
Let $A\colon \widehat{\mathcal{M}}_{p}(\R^n) \rightarrow \R^q$ be a linear mapping. We say that the function $\phi:\R^n \rightarrow \R^q$ is a representative of $A$ if for all $m \in \widehat{\mathcal{M}}_p(\R^n)$, the integral $\int_{\R^n} \phi(x) \dd m(x)$ is well-defined and equal to $Am$:
\begin{equation*}
Am= \int_{\R^n} \phi(x) \dd m(x).
\end{equation*}
If $\phi$ is a representative of $A$, then for any constant $c \in \R^q$, $\phi+c$ is also a representative of $A$, since
\begin{equation*}
\int_{\R^n} c \dd m(x) = 0, \quad \forall m \in \widehat{\mathcal{M}}_p(\R^n).
\end{equation*}
Conversely, if the value of a representative $\phi$ is fixed for a given point $x_0 \in \R^n$, then for all $x \in \R^n$, the value of $\phi(x)$ is determined by
\begin{equation*}
\phi(x)= \phi(x_0) + A(\delta_x-\delta_{x_0}),
\end{equation*}
where $\delta_{x}$ and $\delta_{x_0}$ are the Dirac measures centered at $x$ and $x_0$, respectively.
Therefore, the representative, if it exists, is uniquely defined up to a constant.

\begin{assumption} \label{hypDiff}
\begin{enumerate}
\item For all $m$, $m_1$, and $m_2 \in \bar{B}_p(R)$, there exists a linear form $DF(m)\colon$ $\widehat{\mathcal{M}}_{p}(\R^n) \rightarrow \R$ such that for all $\varepsilon>0$, there exist $\bar{\xi}$ and $\bar{\zeta}$ in $(0,1]$ such that for all $\xi \in [0,\bar{\xi}]$ and for all $\zeta \in [0,\bar{\zeta}]$,
\begin{equation*} \label{eqDirectionalDerivative}
\big| F( \tilde{m} )
- \big[ F(m) + \xi DF(m)(m_1-m) + \zeta DF(m)(m_2-m) \big] \big| \leq \varepsilon(\xi + \zeta),
\end{equation*}
where
\begin{equation*}
\tilde{m}= (1-\xi-\zeta)m + \xi m_1 + \zeta m_2 = m + \xi(m_1-m) + \zeta (m_2-m).
\end{equation*}
Moreover, $DF(m)$ possesses a continuous representative, dominated by $|x|^p$, and denoted $x \in \R^n \mapsto DF(m,x) \in \R$.
\item For all $m$, $m_1$, and $m_2 \in \bar{B}_p(R)$ there exists a linear mapping $DG(m)\colon$ $\widehat{\mathcal{M}}_{p}(\R^n) \rightarrow \R^N$ such that for all $\varepsilon > 0$, there exists $\bar{\xi}$ and $\bar{\zeta}$ in $(0,1]$ such that for all $\xi_a$ and $\xi_b \in [0,\bar{\xi}]$ and for all $\zeta_a$ and $\zeta_b \in [0,\bar{\zeta}]$,
\begin{align}
\big| G( m_b ) - G( m_a ) & -
(\xi_b-\xi_a)DG(m)(m_1-m) 
\notag \\[0.5em]
& \quad
-(\zeta_b-\zeta_a) DG(m)(m_2-m)  \big| \leq \varepsilon \big( |\xi_b-\xi_a| + |\zeta_b-\zeta_a| \big),
\label{eqDirDer2}
\end{align}
where
\begin{align*}
m_a= \ & (1-\xi_a-\zeta_a)m + \xi_a m_1 + \zeta_a m_2 \\
m_b= \ & (1-\xi_b-\zeta_b)m + \xi_b m_1 + \zeta_b m_2.
\end{align*}
Moreover, $DG(m)$ possesses a continuous representative, dominated by $|x|^p$, and denoted $x \in \R^n \mapsto DG(m,x) \in \R^N$.
\end{enumerate}
\end{assumption}

In the article, we make use of the derivative $DF(m)$ (a linear form from $\widehat{\mathcal{M}}_{p}(\R^n)$ to $\R$) and its representative. The two notions can be distinguished according to the presence (or not) of the variable $x$.
Note also that the differentiability assumption on $G$ is a strict differentiability assumption. It is a little bit stronger than the assumption on $F$.

\subsection{Discussion of the notion of derivative} \label{subsection:disscussion}

A general class of cost functions satisfying Assumptions \ref{hypContinuity} and \ref{hypDiff} can be described as follows. Let $K \in \mathbb{N}$, let $\Psi \colon \R^K \rightarrow \R^q$ be differentiable, let $\phi\colon \R^n \rightarrow \R^K$ be a continuous function dominated by $|x|^{p}$.
We define then on $\mathcal{P}_{p}(\R^n)$:
\begin{equation} \label{eqExample1}
H(m)= \Psi \Big( \int_{\R^n} \phi(x) \dd m(x) \Big).
\end{equation}
Note that for all control processes $u \in \mathcal{U}_0(Y_0)$,
\begin{equation*}
H \big( m_T^{0,Y_0,u} \big)= \Psi \Big( \mathbb{E} \big[ \phi \big( X_T^{0,Y_0,u} \big) \big] \Big).
\end{equation*}
For all $R \geq 0$, the continuity of $H$ on $\bar{B}_{p}(R)$  follows from Lemma \ref{lemmaContinuityDominatedCost} (in the appendix).
One can easily check that the mapping $H$ is differentiable in the sense of Assumption \ref{hypDiff}.1. The representative of its derivative is given by
\begin{equation} \label{eqDiffFunctionExp}
DH(m,x)= D \Psi \Big( \int_{\R^n} \phi(z) \dd m(z) \Big) \phi(x),
\end{equation}
up to a constant.
Furthermore, if $\Psi$ is continuously differentiable, then $H$ is differentiable in the sense of Assumption \ref{hypDiff}. Note that the function $\phi$ does not need to be differentiable.
Further examples are discussed in detail in \cite[Section 4]{Pfe15b}.
We finish this subsection with two remarks.

\begin{remark}
The fact that $F$ should be defined on the whole space $\mathcal{P}_{p}(\R^n)$ discards cost functions whose formulation is based on the density of the probability measure (since a density does not always exists, for probability distributions in $\mathcal{P}_p(\R^n)$). For example, the following problem does not fit to the proposed framework:
\begin{equation*}
\inf_{u \in \mathcal{U}_0(Y_0)} \frac{1}{2} \| f - f_{\text{ref}} \|_{L^2(R^n)}, \quad
\text{subject to: $f= \text{PDF}(X_T^{0,Y_0,u})$},
\end{equation*}
where $\text{PDF}$ stands for probability density function and where $f_{\text{ref}}$ is a given probability density function.
\end{remark}

\begin{remark}
The notion of derivative provided in \cite[Section 6]{Car12} and the one introduced in Assumption \ref{hypDiff} are of different nature, because they aim at evaluating the variation of functions from $\bar{B}_p(R)$ to $\R$ on different kinds of paths. While our derivative is represented by a function from $\R^n$ to $\R$, the one of \cite{Car12} is represented by a function from $\R^n$ to $\R^n$ (see \cite[Theorem 6.5]{Car12}).
This difference of nature can be better understood by considering the mapping: $m \in \mathcal{P}_{p}(\R^n) \mapsto \int_{\R^n} \phi \dd m$. This mapping is a monomial, according to the terminology given in \cite[Example, page 43]{Car12}. Its derivative (in the sense of \cite{Car12}) is represented by the mapping: $x \in \R^n \mapsto D\phi(x) \in \R^n$ (see \cite[Example, page 44]{Car12}). In the current framework, the derivative of $m \in \mathcal{P}_{p}(\R^n) \mapsto \int_{\R^n} \phi \dd m$ is the real-valued mapping $x \in \R^n \mapsto \phi(x) \in \R$, up to a constant.
\end{remark}

\subsection{Existence of a solution}

Observe that problem \eqref{eqPb} can be equivalently formulated as follows:
\begin{equation*} \label{eqEquivalentPb} \tag{$P'$}
\inf_{m \in \mathcal{R}} F(m), \quad \text{subject to: } G(m) \leq 0,
\end{equation*}
where $\mathcal{R}$ is defined by \eqref{eqDefR}.
Indeed, if $\bar{u} \in \mathcal{U}_0(Y_0)$ is a solution to \eqref{eqPb}, then $\bar{m}:= m_T^{0,Y_0,\bar{u}}$ is a solution to \eqref{eqEquivalentPb} and conversely, if $\bar{m} \in \mathcal{R}$ is a solution to \eqref{eqEquivalentPb}, then any $\bar{u} \in \mathcal{U}_0(Y_0)$ such that $\bar{m}= m_T^{0,Y_0,\bar{u}}$ is a solution to \eqref{eqPb}.
The feasible set $\mathcal{R}_\text{ad}$ of Problem \eqref{eqEquivalentPb} is defined by
\begin{equation} \label{eq:defRad}
\mathcal{R}_\text{ad}= \{ m \in \mathcal{R} \,|\, G(m) \leq 0 \}.
\end{equation}
By continuity of $F$ for the $d_1$-distance, the value of the following problem:
\begin{equation*} \label{eqMinimizerReformulatedPb} \tag{$P''$}
\inf_{m \in \text{cl}(\mathcal{R}_{\text{ad}})} F(m)
\end{equation*}
is the same as the one of problems \eqref{eqPb} and \eqref{eqEquivalentPb}. Indeed, problem \eqref{eqMinimizerReformulatedPb} is simply obtained by replacing the feasible set of \eqref{eqEquivalentPb} by its closure (for the $d_1$-distance).

Lemma \ref{lemmaLipschitzEstimates} enables us to prove the existence of a solution to problem \eqref{eqMinimizerReformulatedPb}.

\begin{lemma} \label{lemmaExistenceSol}
If $\mathcal{R}_{\text{\emph{ad}}}$ is non-empty, then problem \eqref{eqMinimizerReformulatedPb} has a solution.
\end{lemma}

\begin{proof}
It is proved in Lemma \ref{lemmaCompactnessProperty} that the set $\bar{B}_{p}(R)$ is compact for the $d_1$-distance.
By Lemma \ref{lemmaLipschitzEstimates}, $\mathcal{R} \subseteq \bar{B}_p(R)$ and therefore, $\mathcal{R}_{\text{ad}} \subseteq \bar{B}_p(R)$. Since $\bar{B}_p(R)$ is closed, $\text{cl}(\mathcal{R}_{\text{ad}}) \subseteq \bar{B}_p(R)$ and therefore, $\text{cl}(\mathcal{R}_{\text{ad}})$ is compact, since it is a closed set of a compact set. The existence of a solution follows, since $F$ is continuous.
\end{proof}

It is in general difficult to prove the existence of a solution to \eqref{eqPb}.

\section{Convexity of the reachable set} \label{sectionConvexity}

This section is dedicated to the proof of the convexity of the closure of $\mathcal{R}$ (the set of reachable probability distributions at time $T$). This result is an important tool for the proof of the optimality conditions in Section \ref{sectionOptimality} and for the numerical method developed in Section \ref{sectionNumerics}.
Let us explain the underlying purpose with a simple example. Consider two processes $u_1$ and $u_2 \in \mathcal{U}_0(Y_0)$ and the corresponding final probability distributions $m_T^{0,Y_0,u_1}$ and $m_T^{0,Y_0,u_2}$. We aim at building a control process $u$ such that
\begin{equation} \label{eqBasicRelaxation}
m_T^{0,Y_0,u}= \frac{1}{2} \big( m_T^{0,Y_0,u_1} + m_T^{0,Y_0,u_2} \big).
\end{equation}
A very simple way of building such a control process is to define a random variable $S$ independent of $Y_0$ and $\mathcal{F}_{0,T}$ taking two different values with probability $1/2$. A control $u$ realizing \eqref{eqBasicRelaxation} can then be constructed in $\mathcal{U}_0\big( (S,Y_0) \big)$: it suffices that $u=u_1$ for one value of $S$ and that $u= u_2$ for the other. The obtained controlled process can be seen as a relaxed control, since it is now measurable with respect to a larger filtration.
The main idea of Lemma \ref{lemmaConvexity} is to construct control processes in $\mathcal{U}_0(Y_0)$ imitating the behaviour of the relaxed control process $u$.

\begin{lemma} \label{lemmaConvexity}
The closure of the set of reachable probability measures for the $d_1$-distance, denoted $\text{\emph{cl}}( \mathcal{R})$, is convex.
\end{lemma}

\begin{proof}
Our approach mainly consists in proving that
\begin{equation} \label{eqInclusion1}
\text{conv}(\mathcal{R})
\subseteq \text{cl}(\mathcal{R}).
\end{equation}
Let $K \in \mathbb{N}^*$, let $u^1,...,u^K$ in $\mathcal{U}_0(Y_0)$, let $\theta_1,...,\theta_K$ in $\R_+ \backslash \{ 0 \}$ with $\sum_{k=1}^K \theta_k= 1$. To prove \eqref{eqInclusion1}, it suffices to prove that there exists a sequence $(u^\varepsilon)_{\varepsilon \geq 0}$ in $\mathcal{U}_0(Y_0)$ such that
\begin{equation} \label{eqInclusion1Bis}
d_1 \Big( m_T^{0,Y_0,u^\varepsilon}, {\textstyle \sum_{k=1}^K} \theta_k m_T^{0,Y_0,u^k} \Big) \underset{\varepsilon \to 0}{\longrightarrow} 0.
\end{equation}
Let $0<\varepsilon<T$, let $\tilde{u}^1,...,\tilde{u}^K$ be $K$ processes in $\mathcal{U}_{\varepsilon}(Y_0)$ such that for all $k$, the processes $(u_s^k)_{s \in [0,T-\varepsilon]}$ and $(\tilde{u}_s^k)_{s \in [\varepsilon,T]}$ can be seen as the same measurable function of respectively
\begin{equation*}
(Y_0,(W_s-W_0)_{s \in [0,T-\varepsilon]}) \quad \text{and} \quad (Y_0,(W_s-W_{\varepsilon})_{s \in [\varepsilon,T]}).
\end{equation*}
In other words, we simply delay the observation of the variation of the Brownian motion of a time $\varepsilon$.
Let $-\infty = r_0 < ... < r_K= + \infty$ be such that for all $k=1,...,K$,
\begin{equation*}
\int_{r_{k-1}}^{r_k} e^{-z^2} \dd z= \sqrt{2\pi} \theta_k
\end{equation*}
and let us denote by $A_k$ the following event:
\begin{equation*}
(W_{\varepsilon}^1-W_0^1)/\sqrt{\varepsilon} \in (r_{k-1},r_k),
\end{equation*}
where $W^1$ is the first coordinate of the Brownian motion. For all $k$, we have $\mathbb{P}\big[ A_k \big]= \theta_k$. Fixing $u^0 \in U$, we define $u^\varepsilon \in \mathcal{U}_0(Y_0)$ as follows:
\begin{align*}
u_t^\varepsilon = u^0, & \quad \text{for a.\,e.\@ $t \in (0,\varepsilon)$}, \\
u_t^\varepsilon = \tilde{u}_t^k, & \quad \text{for a.\,e.\@ $t\in (\varepsilon,T)$, when $A_k$ is realized}.
\end{align*}
For all $\phi \in \text{1-Lip}(\R^n)$,
\begin{align}
& \mathbb{E}\big[ \phi \big( X_T^{0,Y_0,u^\varepsilon} \big) \big]
= \sum_{k=1}^K \theta_k \, \mathbb{E} \Big[ \phi \Big( X_T^{\varepsilon,X_{\varepsilon}^{0,Y_0,u^\varepsilon},{u}^\varepsilon} \Big) \,\big|\, A_k \Big] \notag \\
& \quad = \sum_{k=1}^K \theta_k \, \mathbb{E} \Big[ \phi \Big( X_T^{\varepsilon,X_{\varepsilon}^{0,Y_0,u^0},\tilde{u}^k} \Big) \,\big|\, A_k \Big]
= \sum_{k=1}^K \theta_k \Big[ a_k + b_k + \mathbb{E} \big[ \phi\big(X_T^{0,Y_0,u^k} \big) \big] \Big],
\label{eqLemmaConv1}
\end{align}
where $a_k$ and $b_k$ are given by
\begin{align*}
a_k=\ & \mathbb{E} \Big[ \phi \Big( X_T^{\varepsilon,X_{\varepsilon}^{0,Y_0,u^0},\tilde{u}_k} \Big)-\phi \Big( X_T^{\varepsilon,Y_0,\tilde{u}_k} \Big) \big| A_k \Big], \\
b_k=\ & \mathbb{E} \big[ \phi \big( X_T^{\varepsilon, Y_0, \tilde{u}^k} \big) | A_k \big] - \mathbb{E} \big[ \phi \big( X_T^{0,Y_0,u^k} \big) \big].
\end{align*}
Let us first estimate $a_k$. Using the Lipschitz-continuity of $\phi$, we obtain that
\begin{equation*}
\theta_k |a_k| \leq
\mathbb{P}\big[ A_k \big] \mathbb{E} \Big[ \, \Big| X_T^{\varepsilon,X_{\varepsilon}^{0,Y_0,u^0},\tilde{u}_k} - X_T^{\varepsilon,Y_0,\tilde{u}_k} \Big| \big| A_k \Big] \\
\leq \mathbb{E} \Big[ \, \Big| X_T^{\varepsilon,X_{\varepsilon}^{0,Y_0,u^0},\tilde{u}_k} - X_T^{\varepsilon,Y_0,\tilde{u}_k} \Big| \, \Big].
\end{equation*}
We deduce from the Cauchy-Schwarz inequality and Lemma \ref{lemmaLipschitzEstimates} that
\begin{align}
\theta_k |a_k| \leq \ & \mathbb{E} \Big[ \, \Big| X_T^{\varepsilon,X_{\varepsilon}^{0,Y_0,u^0},\tilde{u}_k} - X_T^{\varepsilon,Y_0,\tilde{u}_k} \Big|^2 \, \Big]^{1/2} \notag \\
\leq \ & \sqrt{C_3} \mathbb{E} \big[ | X_{\varepsilon}^{0,Y_0,u^0} - Y_0 |^2 \big]^{1/2} \notag \\
\leq \ & \sqrt{C_2C_3\varepsilon} \, \big( 1 + \mathbb{E}\big[ |Y_0|^2 \big] \big)^{1/2}.
\label{eqLemmaConv2}
\end{align}
Let us estimate $b_k$.
Since $\tilde{u}_k$ and $Y_0$ are independent of $A_k$ and using the definition of $\tilde{u}_k$, we obtain that
\begin{equation*}
\mathbb{E} \big[ \phi(X_T^{\varepsilon,Y_0,\tilde{u}_k}) \,|\, A_k \big]
= \mathbb{E} \big[ \phi(X_T^{\varepsilon,Y_0,\tilde{u}_k}) \big]
= \mathbb{E} \big[ \phi(X_{T-\varepsilon}^{0,Y_0,u_k}) \big].
\end{equation*}
Therefore,
\begin{equation*}
b_k = \mathbb{E} \big[ \phi \big( X_{T-\varepsilon}^{0, Y_0, {u}^k} \big) - \phi \big( X_T^{0,Y_0,u^k} \big) \big].
\end{equation*}
We obtain with the Lipschitz-continuity of $\phi$, the Cauchy-Schwarz inequality, and Lemma \ref{lemmaLipschitzEstimates} that
\begin{align}
|b_k| \leq \ & \mathbb{E} \big[ \big| X_T^{0,Y_0,u^k} - X_{T-\varepsilon}^{0,Y_0,u^k} \big| \big] \notag \\
\leq \ & \mathbb{E} \big[ \big| X_T^{0,Y_0,u^k} - X_{T-\varepsilon}^{0,Y_0,u^k} \big|^2 \big]^{1/2} \notag \\
\leq \ & \sqrt{C_2 \varepsilon} \big( 1 + \mathbb{E}\big[ |Y_0|^2 \big)^{1/2}. \label{eqLemmaConv3}
\end{align}
Combining \eqref{eqLemmaConv1}, \eqref{eqLemmaConv2}, and \eqref{eqLemmaConv3}, we obtain that
\begin{equation*}
\Big| \mathbb{E}\big[ \phi \big( X_T^{0,Y_0,u^\varepsilon} \big) \big]
- \sum_{k=1}^K \theta_k \mathbb{E} \big[ \phi\big(X_T^{0,Y_0,u^k} \big) \big] \Big|
\leq \sum_{k=1}^K \theta_k \big( |a_k| + |b_k| \big)
\leq M \sqrt{\varepsilon},
\end{equation*}
where $M$ is a constant independent of $\phi$ and $\varepsilon$.
Using the dual representation of $d_1$ (given by \eqref{eqDualWasserstein}), we deduce that
\begin{align*}
& d_1 \Big( m_T^{0,Y_0,u^\varepsilon}, { \sum_{k=1}^K} \theta_k m_T^{0,Y_0,u^k} \Big) \\
& \qquad \leq \sup_{\phi \in \text{1-Lip}(\R^n)} \Big\{ \mathbb{E}\big[ \phi \big( X_T^{0,Y_0,u^\varepsilon} \big) \big] - \sum_{k=1}^K \theta_k \mathbb{E} \big[ \phi \big( X_{T}^{0,Y_0,u^k} \big) \big] \Big\}= M\sqrt{\varepsilon}.
\end{align*}
This proves \eqref{eqInclusion1Bis} and thus justifies \eqref{eqInclusion1}. We can now conclude the proof. It follows from \eqref{eqInclusion1} that
\begin{equation} \label{eqInclusion2}
\text{cl} \big[ \text{conv} ( \mathcal{R} ) \big] \subseteq
\text{cl}( \mathcal{R} ).
\end{equation}
Since $\mathcal{R} \subseteq \text{conv}( \mathcal{R})$, we have
$\text{cl} (\mathcal{R}) \subseteq \text{cl} \big[ \text{conv}(\mathcal{R}) \big]$,
and therefore by \eqref{eqInclusion2},
$\text{cl}(\mathcal{R}) = \text{cl} \big[ \text{conv}( \mathcal{R}) \big]$.
It remains to prove that $\text{cl} \big[ \text{conv}(\mathcal{R}) \big]$ is convex, which is an easy task.
\end{proof}

\section{Optimality conditions} \label{sectionOptimality}

We prove in this section the main result: if a control $\bar{u}$ is a solution to \eqref{eqPb} and satisfies a qualification condition, then it is the solution to a standard problem of the form \eqref{eqStandardPb}.
Before proving our result, we recall in Subsection \ref{subSecStandardPb} some well-known properties of the value function associated with a standard problem.

\subsection{Standard problems} \label{subSecStandardPb}

Let $\phi:\R^n \rightarrow \R$ be a continuous function dominated by $|x|^{p}$. Let us define:
\begin{equation*}
\Phi: m \in \mathcal{P}_{p}(\R^n) \mapsto \int_{\R^n} \phi(x) \dd m(x).
\end{equation*}
The mapping $\Phi$ is linear, in so far as for all $m_1$ and $m_2 \in \mathcal{P}_{p}(\R^n)$, for all $\theta \in [0,1]$,
\begin{equation*}
\Phi(\theta m_1 + (1-\theta)m_2)= \theta \Phi(m_1) + (1-\theta) \Phi(m_2).
\end{equation*}
It is also continuous for the $d_1$-distance, see Lemma \ref{lemmaContinuityDominatedCost} (in the appendix).
We denote by \eqref{eqLinearizedPb} the following standard problem:
\begin{equation*} \label{eqLinearizedPb} \tag{$P(\phi)$}
\inf_{u \in \mathcal{U}_0(Y_0)} \, \mathbb{E}\big[ \phi(X_T^{0,Y_0,u}) \big].
\end{equation*}
Let $\hat{u} \in \mathcal{U}_0(Y_0)$ and let $\hat{m}= m_T^{0,Y_0,\hat{u}}$. By continuity of $\Phi$, the control process $\hat{u}$ is a solution to \eqref{eqLinearizedPb} if and only if
\begin{equation*}
\Phi(\hat{m})= \inf_{m \in \text{cl}(\mathcal{R})} \Phi(m).
\end{equation*}

We recall, for future reference, some well-known results concerning the value function associated with the standard problem \eqref{eqLinearizedPb}. We refer to the textbooks \cite{FS06}, \cite{Pha09}, \cite{YZ99} on this topic.
The value function $V$ associated with \eqref{eqLinearizedPb} is defined for all $t \in [0,T]$ and for all $x \in \R^n$ by
\begin{equation*}
V(t,x)= \inf_{u \in U_t^0} \mathbb{E}\big[ \phi(X_T^{t,x,u}) \big].
\end{equation*}
It can be characterized as the unique viscosity solution to the following Hamilton-Jacobi-Bellman (HJB) equation:
\begin{equation} \label{eqHJB}
-\partial_t V(t,x)= \inf_{u \in U} \Big\{ H(x,u,\partial_x V(t,x),\partial_{xx}^2 V(t,x)) \Big\}, \quad V(T,x)= \phi(x),
\end{equation}
where the Hamiltonian $H$ is defined  for $x \in \R^n$, $u \in U$, $p \in \R^n$, and $Q \in \R^{n \times n}$ by
\begin{equation*}
H(x,u,p,Q)=  \langle p, b(x,u) \rangle + \frac{1}{2} \text{tr}( \sigma(x,u) \sigma(x,u)^{\top} Q ).
\end{equation*}
If $V$ is sufficiently smooth, one can prove with a verification argument that any control process $u$ is a global solution to \eqref{eqLinearizedPb} if almost surely,
\begin{equation*}
u_t \in \underset{v \in U}{\text{arg min}} \ \Big\{ H(X_t^{0,Y_0,u},v,\partial_x V(t,X_t^{0,Y_0,u}),\partial_{xx}^2 V(t,X_t^{0,Y_0,u})) \Big\}, \quad \text{for a.e.\@ $t \in (0,T)$}.
\end{equation*}
Finally, note that the value of problem \eqref{eqLinearizedPb} is given by
\begin{equation*}
\inf_{u \in \mathcal{U}_0(Y_0)} \, \mathbb{E}\big[ \phi(X_T^{0,Y_0,u}) \big] = \int_{\R^n} V(0,x) \dd m(x),
\end{equation*}
where $m= \mathcal{L}(Y_0)$.

\subsection{Main result}

In this subsection, we give first-order optimality conditions in variational form for problem \eqref{eqPb} (defined in the introduction, page \pageref{eqPb}).
Along this subsection, a solution $\bar{u} \in \mathcal{U}_0(Y_0)$ to problem \eqref{eqPb} is fixed. We also set
\begin{equation*}
\bar{m}= m_T^{0,Y_0,\bar{u}}.
\end{equation*}
We first give a metric regularity result (Theorem \ref{theoMetricReg}), which is a key tool for the proof of the optimality conditions (Theorem \ref{theoOptiCond}).

Let us consider the sets $A$ and $I$ of active and inactive constraints at $\bar{m}$, defined by
\begin{equation*}
A= \{ j=1,...,N \,|\, G_j(\bar{m})= 0 \} \quad \text{and} \quad
I= \{ j=1,...,N \,|\, G_j(\bar{m})< 0 \}.
\end{equation*}
Let $N_A$ be the cardinality of $A$.
We define
\begin{equation*}
G_A(m)= (G_j(m))_{j \in A} \in \R^{N_A} \quad \text{and} \quad G_I(m)= (G_j(m))_{j \in I} \in \R^{N-N_A}.
\end{equation*}
We have
\begin{equation*}
\mathcal{R}_{\text{ad}} = \big\{ m \in \mathcal{R} \,|\, G_{A}(m) \leq 0, \ G_{I}(m) \leq 0 \big\}, \quad
G_{A}(\bar{m}) = 0, \quad \text{and} \quad
G_{I}(\bar{m}) < 0.
\end{equation*}

The following assumption is a qualification condition.

\begin{assumption} \label{hypQualification}
There exists $m_0 \in \text{\emph{cl}}(\mathcal{R})$ such that $DG_{A}(\bar{m})(m_0-\bar{m})<0$.
\end{assumption}

For all $z \in \R^{N_A}$, we denote by $z_+$ the vector defined by $(z_+)_j= \max(z_j,0)$ for $j \in A$. One can easily check that for all $z$ and $\tilde{z} \in \R^{N_A}$,
\begin{equation} \label{eqPropSupNorm}
| z_+ |_\infty \leq | z |_\infty, \quad
z- z_+ \leq 0, \quad \text{and} \quad
\big| | \tilde{z}_+|_\infty - | z_+ |_\infty \big| \leq | \tilde{z} - z |_\infty.
\end{equation}

\begin{theorem} \label{theoMetricReg}
If Assumption \ref{hypQualification} holds, then for all $m \in \text{\emph{cl}}(\mathcal{R})$, there exist two constants $\bar{\theta} \in (0,1]$ and $C>0$ such that for all $\theta \in [0,\bar{\theta}]$, for all $\varepsilon >0$, there exists $\eta \in [0,1]$ such that
\begin{equation} \label{eqRegMet}
\eta \leq C \, \big| G_{A}( m_\theta )_+ \big|_\infty + \varepsilon
\end{equation}
and such that $G\big( (1-\eta)m_\theta + \eta m_0 \big) < 0$,
where $m_\theta= (1-\theta) \bar{m} + \theta m$ and where $m_0$ is given by Assumption \ref{hypQualification}.
\end{theorem}

The estimate \eqref{eqRegMet} is basically an estimate of the distance of $m_\theta$ to $\text{cl}(\mathcal{R}_{\text{ad}})$. Indeed, by the convexity of $\text{cl}(\mathcal{R})$, the probability measure $(1-\eta)m_\theta + \eta m_0$ lies in $\text{cl}(\mathcal{R})$. Since $G$ is continuous and since $G\big((1-\eta)m_\theta + \eta m_0 \big) < 0$, the probability measure $(1-\eta)m_\theta + \eta m_0$ lies in $\text{cl}(\mathcal{R}_{\text{ad}})$. It is at a distance $\eta d_1(m_\theta,m_0)$ of $m_\theta$. The real number $\eta \geq 0$ is of same order as the quantity $|G_A(m_\theta)_+|_\infty$, which indicates how much the constraints are violated.

\begin{proof}[Proof of Theorem \ref{theoMetricReg}]
For all $\theta \in [0,1]$, we define
\begin{equation*}
\mathcal{G}_\theta \colon \eta \in [0,1] \mapsto
\mathcal{G}_\theta(\eta) = G_{A} \big( (1-\eta)m_\theta + \eta m_0 \big) \in \R^{N_A}.
\end{equation*}
By Assumption \ref{hypQualification}, $DG_A(\bar{m})(m_0-\bar{m}) < 0$. Let $\alpha > 0$ be such that
\begin{equation} \label{eqRM0}
DG_A(\bar{m})(m_0-\bar{m}) \leq - \alpha \mathbf{1},
\end{equation}
where $\mathbf{1}=(1,...,1) \in \R^{N_A}$. The above inequality (as well as all those involving vectors) must be understood coordinatewise.

\emph{Claim 1}. There exist $\bar{\eta} \in (0,1]$ and $\bar{\theta} \in (0,1]$ such that for all $\eta \in [0,\bar{\eta}]$ and for all $\theta \in [0,\bar{\theta}]$,
\begin{equation} \label{eqRM1}
\mathcal{G}_\theta(\eta)-\mathcal{G}_\theta(0) \leq -\frac{\alpha \eta}{2} \mathbf{1}.
\end{equation}
Let us prove this claim. Let $\varepsilon= \alpha/8$, let $\bar{\xi}$ and $\bar{\zeta}$ in $(0,1]$ be such that \eqref{eqDirDer2} holds. We set $\bar{\eta}= \bar{\xi}$ and $\bar{\theta}= \bar{\zeta}$.
We reduce the value of $\bar{\theta}$, if necessary, so that
\begin{equation} \label{eqReg111}
\bar{\theta} |DG_A(\bar{m})(m-\bar{m})|_\infty \leq \frac{\alpha}{4}.
\end{equation}
For all $\eta \in [0,\bar{\eta}]$ and for all $\theta \in [0,\bar{\theta}]$, we have
\begin{align}
(1-\eta) m_\theta + \eta m_0
= \ & (1-\eta) (1-\theta) \bar{m} + (1-\eta) \theta m + \eta m_0, \notag \\
(1-\eta) m_\theta + \eta m_0 - m_\theta
= \ & -\eta \theta (m-\bar{m}) + \eta (m_0-\bar{m}), \label{eqCombiForReg}
\end{align}
therefore, combining \eqref{eqDirDer2} and \eqref{eqCombiForReg},
\begin{align} 
& \big| \mathcal{G}_\theta(\eta)-\mathcal{G}_\theta(0)- \eta DG_A(\bar{m})(m_0-\bar{m}) + \eta \theta DG_A(\bar{m})(m-\bar{m}) \big|_\infty \notag \\
& \qquad = \big| \mathcal{G}_\theta(\eta)-\mathcal{G}_\theta(0)
- DG_A(\bar{m})\big( (1-\eta) m_\theta + \eta m_0 - m_\theta) \big) \big|_\infty \notag \\
& \qquad \leq  \frac{\alpha}{8}( \eta \theta + \eta ) \leq \frac{\alpha\eta}{4}.
\label{eqRM2}
\end{align}
Moreover, by \eqref{eqReg111},
\begin{equation} \label{eqRM3}
|\eta \theta DG_A(\bar{m})(m-\bar{m})|_\infty
\leq \eta \bar{\theta} |DG_A(\bar{m})(m-\bar{m})|_\infty
\leq \frac{\alpha \eta}{4}.
\end{equation}
Combining \eqref{eqRM2} and \eqref{eqRM3}, we obtain that
\begin{equation*}
| \mathcal{G}_\theta(\eta)-\mathcal{G}_\theta(0) - \eta DG_A(\bar{m})(m_0-\bar{m}) |_\infty
\leq |\eta \theta DG_A(\bar{m}) (m-\bar{m})|_\infty + \frac{\alpha \eta}{4}
\leq \frac{\alpha \eta}{2}.
\end{equation*}
It follows that
\begin{equation*}
\mathcal{G}_\theta(\eta)-\mathcal{G}_\theta(0) \leq \eta DG_A(\bar{m})(m_0-\bar{m}) + \frac{\alpha \eta}{2} \mathbf{1}.
\end{equation*}
Combining the above inequality with \eqref{eqRM0}, we obtain \eqref{eqRM1}. The claim is proved.

Now, let $\eta_0 \in (0,1]$ and $\theta_0 \in (0,1]$ be sufficiently small, so that for all $\eta \in [0,\eta_0]$ and for all $\theta \in [0,\theta_0]$,
\begin{equation*}
G_I\big( (1-\eta) m_\theta + \eta m_0 \big) < 0.
\end{equation*}
Let us set
\begin{equation} \label{eqRM4}
\gamma= \frac{\alpha}{4} \min(\bar{\eta},\eta_0).
\end{equation}
Recall that $\mathcal{G}_0(0)=G_A(\bar{m})= 0$.
Once again, we reduce the value of $\bar{\theta}$, if necessary, so that $\bar{\theta} \leq \theta_0$ and so that for all $\theta \in [0,\bar{\theta}]$,
\begin{equation} \label{eqRM5}
| \mathcal{G}_\theta(0) |_\infty \leq \gamma.
\end{equation}

\emph{Claim 2}. There exists $C>0$ such that for all $y \in \R^{N_{A}}$ with $| y |_\infty \leq \gamma$, for all $\theta \in [0,\bar{\theta}]$, there exists $\eta \in [0,\bar{\eta}]$ such that
\begin{equation*}
\mathcal{G}_\theta(\eta) + y  \leq 0 \quad \text{and} \quad
\eta \leq C \big| \big( \mathcal{G}_\theta(0)+y \big)_+ \big|_\infty.
\end{equation*}
Let us prove the claim. We set $C= \frac{2}{\alpha}$.
For a given $\theta \in [0,\bar{\theta}]$ and for a given $y$ such that $| y |_\infty \leq \gamma$, we set
$\eta = C | (\mathcal{G}_\theta(0)+y)_+ |_\infty$.
Then, to prove the claim, we just have to check that $\eta \leq \bar{\eta}$ and that $\mathcal{G}_\theta(\eta) + y \leq 0$.
Using \eqref{eqRM4}, \eqref{eqRM5}, and the definition of $C$, we obtain that
\begin{equation*}
| (\mathcal{G}_\theta(0) + y)_+ |_\infty
\leq | \mathcal{G}_\theta(0) + y |_\infty
\leq | \mathcal{G}_\theta(0) |_\infty + | y |_\infty
\leq 2 \gamma
\leq \frac{\alpha \bar{\eta}}{2}
= \frac{\bar{\eta}}{C}.
\end{equation*}
Therefore, $\eta = C | (\mathcal{G}_\theta(0) + y)_+ |_\infty \leq \bar{\eta}$. Using the first claim, we obtain that
\begin{equation} \label{eqRM6}
\mathcal{G}_\theta(\eta) + y
\leq \mathcal{G}_\theta(0) - \frac{\alpha \eta}{2} \mathbf{1} + y.
\end{equation}
It directly follows from the definitions of $C$ and $\eta$ that
\begin{equation*}
- \frac{\alpha \eta}{2} \mathbf{1} = - \frac{\eta}{C} \mathbf{1} \leq
- |(\mathcal{G}_\theta(0)+y)_+|_\infty \mathbf{1}
\leq
-(\mathcal{G}_\theta(0) + y)_+.
\end{equation*}
Combined with \eqref{eqRM6}, we obtain that
\begin{equation*}
\mathcal{G}_\theta(\eta) + y \leq (\mathcal{G}_\theta(0) + y) -(\mathcal{G}_\theta(0) + y)_+
\end{equation*}
and finally that $\mathcal{G}_\theta(\eta) + y \leq 0$, by \eqref{eqPropSupNorm}.
This proves the second claim.

\emph{Conclusion}. We can finally prove the theorem. Let $\theta \in [0,\bar{\theta}]$ and let $\varepsilon > 0$.
We set
\begin{equation*}
y= \min \Big( \gamma, \frac{\varepsilon}{C} \Big) \mathbf{1}.
\end{equation*}
Since $| y |_\infty \leq \gamma$, there exists $\eta \in [0,\bar{\eta}]$ such that $\mathcal{G}_\theta(\eta) + y \leq 0$ and such that $\eta \leq C | (\mathcal{G}_\theta(0)+y)_+ |_\infty$, by the second claim.
By \eqref{eqPropSupNorm},
\begin{equation} \label{eqOneMoreEstimate}
\eta \leq C | (\mathcal{G}_\theta(0)+y)_+ |_\infty
\leq C \big( | \mathcal{G}_\theta(0)_+ |_\infty + | y |_\infty \big).
\end{equation}
Since $\mathcal{G}_\theta(0) = G_A(m_\theta)$ and $| y |_\infty \leq \varepsilon/C$, we obtain that
\begin{equation*}
\eta \leq C | G_A(m_\theta)_+ |_\infty + \varepsilon.
\end{equation*}
The estimate \eqref{eqRegMet} is therefore satisfied.
Since $| \mathcal{G}_\theta(0)_+ |_\infty \leq \gamma$ and $| y |_\infty \leq \gamma$, we deduce from \eqref{eqRM4} and from \eqref{eqOneMoreEstimate} that
\begin{equation} \label{eqRegM5}
\eta \leq 2C \gamma = \frac{4\gamma}{\alpha} \leq \eta_0.
\end{equation}
Let us set $\hat{m}= (1-\eta)m + \eta m_0$. By construction, $\mathcal{G}_\theta(\eta) + y = G_{A}(\hat{m}) + y \leq 0$ and thus,
$G_{A}(\hat{m}) < 0$, since $y> 0$.
Moreover, $G_I(\hat{m}) <0$, since $\theta \in [0,\theta_0]$ and $\eta \in [0,\eta_0]$. Therefore, $G(\hat{m})<0$. The theorem is proved.
\end{proof}

In the following theorem, we prove first-order optimality conditions in variational form for problem \eqref{eqPb}. We make use of the Lagrangian $L$, defined by
\begin{equation*}
L\colon (m,\lambda) \in \bar{B}_p(R) \times \R^N \mapsto F(m) + \langle \lambda, G(m) \rangle.
\end{equation*}
The Lagrangian is differentiable (with respect to $m$) in the sense of Assumption \ref{hypDiff}.1 with
\begin{equation*}
DL(m,\lambda)\tilde{m} = DF(m) \tilde{m} + \langle \lambda, DG(m)\tilde{m} \rangle , \quad \forall \tilde{m} \in \widehat{\mathcal{M}}_p(\R^n).
\end{equation*}
A representative of $\tilde{m} \mapsto DL(m,\lambda)\tilde{m}$ is given by
\begin{equation*}
DL(m,\lambda,x)= DF(m,x) + \langle \lambda, DG(m,x) \rangle,
\end{equation*}
up to a constant. Note that the mapping $\tilde{m} \in \bar{B}_p(R) \mapsto DL(m,\lambda)\tilde{m}$ is continuous, by Lemma \ref{lemmaContinuityDominatedCost}.

In the sequel, we say that a non-negative Lagrange multiplier $\lambda \in \R^N$ satisfies the complementarity condition at $\bar{m}$ if for all $i=1,...,N$, $G_i(\bar{m}) < 0 \Longrightarrow \lambda_i= 0$.

\begin{theorem} \label{theoOptiCond}
Let $\bar{u}$ be a solution to problem \eqref{eqPb}. Let $\bar{m}= m_T^{0,Y_0,\bar{u}}$. If Assumption \ref{hypQualification} holds, then there exists a non-negative Lagrange multiplier $\lambda \in \R^N$ satisfying the complementarity condition at $\bar{m}$ which is such that $\bar{u}$ is a solution to the standard problem \eqref{eqLinearizedPb} with
\begin{equation*}
\phi(x)= DL(\bar{m},\lambda,x).
\end{equation*}
\end{theorem}

\begin{remark}
We say then that the control process $\bar{u}$ satisfies the optimality conditions in variational form.
The optimality of $\bar{u}$ for the standard problem with $\phi(x)= DL(\bar{m},x)$ is equivalent to the following variational inequalities:
\begin{equation*}
\mathbb{E} \big[ DL(\bar{m},\lambda,X_T^{0,Y_0,u} ) ]
\geq \mathbb{E} \big[ DL(\bar{m},\lambda,X_T^{0,Y_0,\bar{u}} ) \big], \quad \forall u \in \mathcal{U}_0(Y_0)
\end{equation*}
and
\begin{equation} \label{eqVarIneqBis}
DL(\bar{m},\lambda)(m-\bar{m}) \geq 0, \quad \forall m \in \text{cl}(\mathcal{R}).
\end{equation}
In the sequel, we say that a probability measure $\bar{m} \in \text{cl}(\mathcal{R})$ such that $G(\bar{m}) \leq 0$ satisfies the optimality conditions in variational form if there exists a multiplier $\lambda \geq 0$ satisfying the complementarity condition and such that \eqref{eqVarIneqBis} holds.
\end{remark}

\begin{proof}[Proof of Theorem \ref{theoOptiCond}]
In view of the complementarity condition, it suffices to prove the existence of $\lambda_A \in \R^{N_A}$ such that $\lambda_A \geq 0$ and such that
\begin{equation*}
\inf_{m \in \text{cl}(\mathcal{R})} D\tilde{L}(\bar{m},\lambda_A)(m-\bar{m}) = 0,
\end{equation*}
where $\tilde{L}(m,\lambda_A)= F(m) + \langle \lambda_A,G_A(m) \rangle$. For all $m \in \text{cl}(\mathcal{R})$,
\begin{equation*}
D\tilde{L}(\bar{m},\lambda_A)(m-\bar{m})= DF(\bar{m})(m-\bar{m}) + \langle \lambda_A, DG_A(\bar{m})(m-\bar{m})\rangle.
\end{equation*}
For all $y \in \R^{N_A}$, we consider the following optimization problem, denoted \eqref{eqLinP}:
\begin{equation*} \label{eqLinP} \tag{$LP(y)$}
V(y)= \inf_{m \in \text{cl}(\mathcal{R})} DF(\bar{m})(m-\bar{m}), \quad \text{subject to: }
DG_A(\bar{m})(m-\bar{m}) + y \leq 0.
\end{equation*}

\emph{Step 1.} We first prove that $V(0)= 0$. For $y= 0$, $\bar{m}$ is feasible (for problem \eqref{eqLinP} with $y=0$), thus $V(0) \leq DF(\bar{m})(\bar{m}-\bar{m})= 0$. Now, let $m \in \text{cl}(\mathcal{R})$ be such that $DG_A(\bar{m})(m-\bar{m}) \leq 0$.
Let $\bar{\theta} > 0$ and $C>0$ be given by Theorem \ref{theoMetricReg}. Let $(\theta_k)_{k \in \mathbb{N}}$ be a convergent sequence with limit 0 taking values in $(0,\bar{\theta}]$.
For all $k$, we set
\begin{equation*}
m_k = (1-\theta_k) \bar{m} + \theta_k m.
\end{equation*}
By Assumption \ref{hypDiff}, we have
\begin{equation*}
G_A(m_k)= G_A(\bar{m}) + \theta_k DG_A(\bar{m})(m-\bar{m}) + o(\theta_k).
\end{equation*}
Since $G_A(\bar{m})= 0$ and $DG_A(\bar{m})(m-\bar{m}) \leq 0$, we have
\begin{equation*}
| G_A(m_k)_+ |_\infty = o(\theta_k).
\end{equation*}
By Theorem \ref{theoMetricReg}, there exists for all $k \in \mathbb{N}$ a real number $\eta_k \in [0,1]$ such that 
\begin{equation*}
\eta_k \leq C | G_A(m_k)_+ |_\infty + \theta_k^2 = o(\theta_k)
\end{equation*}
and such that $G(\hat{m}_k)<0$, where
\begin{align*}
\hat{m}_k= \ & (1-\eta_k) m_k + \eta_k m_0 \notag \\
= \ & (1-\eta_k)(1-\theta_k) \bar{m} + (1-\eta_k)\theta_k m + \eta_k m_0. 
\end{align*}
Thus,
\begin{equation}  \label{eqExpansionOC}
\hat{m}_k- \bar{m}= (1-\eta_k)\theta_k (m-\bar{m}) + \eta_k (m_0-\bar{m}).
\end{equation}
Since $\text{cl}(\mathcal{R})$ is convex (Lemma \ref{lemmaConvexity}), $\hat{m}_k \in \text{cl}(\mathcal{R})$. Therefore, for all $k$, by continuity of $F$ and $G$, there exists $\tilde{m}_k \in \mathcal{R}$ such that $G(\tilde{m}_k) \leq 0$ and such that $F(\tilde{m}_k)-F(\hat{m}_k) \leq \theta_k^2$. Using the differentiability assumption on $F$ (Assumption \ref{hypDiff}), the feasibility of $\tilde{m}_k$, the fact that $\eta_k= o(\theta_k)$ and \eqref{eqExpansionOC}, we obtain that
\begin{align*}
0 \leq \ & \frac{F(\tilde{m}_k)-F(\bar{m})}{\theta_k} \\
\leq \ & \frac{F(\hat{m}_k)-F(\bar{m})}{\theta_k} + \frac{F(\tilde{m}_k)-F(\hat{m}_k)}{\theta_k} \\
= \ & (1-\eta_k) DF(\bar{m})(m-\bar{m}) + \frac{\eta_k}{\theta_k} DF(\bar{m})(m_0-\bar{m}) + o(1) \\
= \ & DF(\bar{m})(m-\bar{m}) + o(1).
\end{align*}
It follows that $DF(\bar{m})(m-\bar{m}) \geq 0$ and finally proves that $V(0)= 0$.

\emph{Step 2.} We compute now the Legendre-Fenchel transform (see \cite[Relation 2.210]{BS00} for a definition) of $V$.
For all $\lambda_A \in \R^{N_A}$, we have
\begin{align*}
V^*(\lambda_A)= \ & \sup_{y \in \R^{N_A}} \big( \langle \lambda_A, y \rangle - V(y) \big) \\
= \ & \sup_{y \in \R^{N_A}} \Big( \langle \lambda_A,y \rangle -  \Big( \inf_{m \in \text{cl}(\mathcal{R})} DF(\bar{m}) (m-\bar{m}),\ \text{subject to: } DG_A(\bar{m})(m-\bar{m}) + y \leq 0 \Big) \Big) \\
= \ & \sup_{\begin{subarray}{c}y \in \R^{N_A}\\ m \in \text{cl}(\mathcal{R}) \end{subarray}} \Big( \langle \lambda_A,y \rangle - DF(\bar{m})(m-\bar{m}), \ \text{subject to: } DG_A(\bar{m}) (m-\bar{m}) + y \leq 0 \Big).
\end{align*}
Using the change of variable $z= DG_A(\bar{m})(m-\bar{m}) + y$, we obtain:
\begin{align*}
V^*(\lambda_A)
= \ & \sup_{\begin{subarray}{c} z \in \R^{N_A} \\ m  \in \text{cl}(\mathcal{R}) \end{subarray}} \Big( \langle \lambda_A, z- DG_A(\bar{m})(m-\bar{m}) \rangle - DF(\bar{m})(m-\bar{m}), \ \text{subject to: } z \leq 0 \Big) \\
= \ & \sup_{m \in \text{cl}(\mathcal{R})} \Big( \Big( \sup_{z \in \R^{N_A}} \langle \lambda_A, z \rangle, \quad \text{subject to: } z \leq 0 \Big) - D\tilde{L}(\bar{m},\lambda_A)(m-\bar{m}) \Big).
\end{align*}
Observing that
\begin{equation*}
\Big( \sup_{z \in \R^{N_A}} \langle \lambda_A, z \rangle, \quad \text{subject to: } z \leq 0 \Big)
= \begin{cases}
\begin{array}{cl}
0 & \text{ if $\lambda_A \geq 0$} \\
+\infty & \text{ otherwise},
\end{array}
\end{cases}
\end{equation*}
we deduce that
\begin{equation} \label{eqConjugate}
V^*(\lambda_A)= \begin{cases}
\begin{array}{cl}
- \inf_{m \in \text{cl}(\mathcal{R})} D \tilde{L}(\bar{m},\lambda_A)(m-\bar{m}) & \text{if $\lambda_A \geq 0$} \\
+ \infty & \text{otherwise}.
\end{array}
\end{cases}
\end{equation}

\emph{Step 3.} Using the convexity of $\text{cl}(\mathcal{R})$ (Lemma \ref{lemmaConvexity}), one can easily show that $V$ is a convex function. Let $\alpha>0$ be such that \eqref{eqRM0} holds. Then, for any $y \in \R^{N_A}$ with $| y |_\infty \leq \alpha$,
\begin{equation*}
DG_A(\bar{m})(m_0-\bar{m}) + y \leq 0
\end{equation*}
and therefore, problem $(LP(y))$ is feasible and $V(y) \leq DF(\bar{m})(m_0-\bar{m})$. 
It follows from \cite[Proposition 2.018, Proposition 2.126]{BS00} that $V$ is continuous in the neighbourhood of 0 and has a non-empty subdifferential $\partial V(0)$ at 0. Let $\lambda_A \in \partial V(0)$, by \cite[Relation 2.232]{BS00}, we have
\begin{equation*}
V^*(\lambda_A)= V(0) + V^*(\lambda) = \langle 0, \lambda \rangle = 0.
\end{equation*}
Thus, by \eqref{eqConjugate}, $\lambda_A \geq 0$ and
\begin{equation*}
\inf_{m \in \text{cl}(\mathcal{R})} D\tilde{L}(\bar{m},\lambda_A)(m-\bar{m}) = 0.
\end{equation*}
The theorem is proved.
\end{proof}

The approach which has been employed to prove Theorem \ref{theoOptiCond} is similar to the one based on relaxation with Young measures for deterministic non-linear optimal control problems. This approach is explained in \cite[Section 3]{BPS13} for example, where Pontryagin's principle is directly deduced from the convexity of the set of reachable linearised states.

The following lemma shows that the value of the standard problem can be used to estimate the loss of optimality of a given probability measure $\hat{m}$ in $\mathcal{R}_{\text{ad}}$ (defined by \eqref{eq:defRad}), when the mappings $F$, $G_1$,...,$G_N$ are convex.
We say that $F$ is convex if for all $\theta \in [0,1]$, for all $m_1$ and $m_2 \in \bar{B}_p(R)$,
\begin{equation} \label{eqConvexityChi}
F(\theta m_1 + (1-\theta) m_2) \leq \theta F(m_1) + (1-\theta) F(m_2).
\end{equation}
The same definition is used for $G_1$,...,$G_N$.
Note that if $F$ is convex, then for all $m_1$ and $m_2 \in \bar{B}_p(R)$,
\begin{equation*}
F(m_2)-F(m_1) \geq DF(m_1)(m_2-m_1).
\end{equation*}

\begin{lemma} \label{lemmaEstimOptimality}
Denote by $\text{\emph{Val}}(P)$ the value of Problem \eqref{eqPb}.
Assume that $F$, $G_1$,...,$G_N$ are convex. Then, for all $\hat{m} \in \mathcal{R}_{\text{ad}}$, for all non-negative $\lambda \in \R^N$ such that the complementarity condition holds at $\hat{m}$, 
the following upper estimate holds:
\begin{equation} \label{eqEstim}
F(\hat{m}) - \text{\emph{Val}}(P) \leq -\inf_{m \in \text{\emph{cl}}(\mathcal{R})} DL(\hat{m},\lambda)(m-\hat{m}).
\end{equation}
\end{lemma}

\begin{proof}
Let $m \in \mathcal{R}_{\text{ad}}$. Since $F$ is convex, we have
\begin{equation} \label{eqEstim1}
F(m)-F(\hat{m}) \geq DF(\hat{m})(m-\hat{m}).
\end{equation}
Denoting by $A$ the active set at $\hat{m}$ and setting
\begin{equation*}
G_A(m) = (G_j(m))_{j \in A} \quad \text{and} \quad \lambda_A= (\lambda_j)_{j \in A},
\end{equation*}
we obtain, using the feasibility of $m$ and the convexity of $G_1$,...,$G_m$ that
\begin{equation*}
0 \geq G_A(m) = G_A(m)-G_A(\hat{m}) \geq DG_A(\hat{m})(m-\hat{m}).
\end{equation*}
Since $\lambda_A \geq 0$, we deduce that
\begin{equation} \label{eqEstim2}
0 \geq \langle \lambda_A,DG_A(\hat{m})(m-\hat{m}) \rangle.
\end{equation}
By the complementarity condition,
\begin{equation} \label{eqEstim3}
\langle \lambda_A,DG_A(\hat{m}) (m-\hat{m}) \rangle
= \langle \lambda, DG(\hat{m})(m-\hat{m}) \rangle.
\end{equation}
Adding \eqref{eqEstim1}, \eqref{eqEstim2}, and \eqref{eqEstim3} together, we obtain that
\begin{equation*}
F(m)-F(\hat{m}) \geq DF(\hat{m})(m-\hat{m}) + \langle \lambda, DG(\hat{m})(m-\hat{m}) \rangle
= DL(\hat{m},\lambda)(m-\hat{m}).
\end{equation*}
Minimizing successively both sides with respect to $m$, we obtain that
\begin{equation*}
\text{Val}(P) - F(\hat{m}) \geq \inf_{m \in \mathcal{R}_{\text{ad}}} DL(\hat{m},\lambda)(m-\hat{m}).
\end{equation*}
Since $\mathcal{R}_{\text{ad}} \subseteq \text{cl}(\mathcal{R})$, we finally obtain that
\begin{equation*}
\text{Val}(P) - F(\hat{m}) \geq \inf_{m \in \text{cl}(\mathcal{R})} DL(\hat{m},\lambda)(m-\hat{m}),
\end{equation*}
which concludes the proof.
\end{proof}

As a corollary, we obtain that the optimality conditions in variational form  are sufficient optimality conditions, in the convex case.

\begin{corollary} \label{coroSuffCond}
Assume that $F$,$G_1$,..,$G_N$ are convex.
Let $\hat{u}$ be a feasible control process satisfying the optimality conditions in variational form.
Then, $\hat{u}$ is a solution to \eqref{eqPb}.
\end{corollary}

\begin{proof}
In this situation, the right-hand side of inequality \eqref{eqEstim} is equal to 0, which directly proves the optimality of $\hat{u}$.
\end{proof}

We finish this section with a corollary dealing with stochastic optimal control problems with an expectation constraint.

\begin{corollary}
Let $f\colon \R^n \rightarrow \R$, $g\colon \R^n \rightarrow \R^N$ be two continuous functions, dominated by $|x|^p$. Assume that there exists $u_0 \in \mathcal{U}_0(Y_0)$ such that $\mathbb{E} \big[ g(X_T^{0,Y_0,u_0}) \big] < 0$. Then, any $u \in \mathcal{U}_0(Y_0)$ is a solution to the following problem:
\begin{equation} \label{eqProbaContr}
\inf_{u \in \mathcal{U}_0(Y_0)} \mathbb{E}\big[ f(X_T^{0,Y_0,u}) \big] \quad \text{subject to: }
\mathbb{E}\big[ g(X_T^{0,Y_0,u}) \big] \leq 0
\end{equation}
if and only if $u$ is feasible and there exists $\lambda \geq 0$ such that $\mathbb{E}\big[ g_i(X_T^{0,Y_0,u}) \big] < 0 \Longrightarrow \lambda_i = 0$, for all $i=1,...,N$, and such that ${u}$ is a solution to \eqref{eqLinearizedPb} with
\begin{equation*}
\phi(x)= f(x) + \langle \lambda, g(x) \rangle.
\end{equation*}
\end{corollary}

\begin{proof}
Setting $F(m)= \int_{\R^n} f(x) \dd m(x)$ and $G(m)= \int_{\R^n} g(x) \dd m(x)$, we obtain that problem \eqref{eqProbaContr} falls into the general class of problems studied in the article. The functions $F$ and $G$ satisfy the required regularity assumptions (see Subsection \ref{subsection:disscussion}). Note that $DF(m,x)= f(x)$ and that $DG(m,x)= g(x)$ (up to a constant). The existence of $u_0$ ensures that the qualification condition is satisfied. The mappings $F$ and $G$ are clearly convex, therefore, the optimality conditions in variational form are necessary and sufficient, by Theorem \ref{theoOptiCond} and Corollary \ref{coroSuffCond}.
\end{proof}

\section{Numerical method and results} \label{sectionNumerics}

\subsection{Augmented Lagrangian Method}

We provide in this section a numerical method for solving problem \eqref{eqPb} and give results for two academical problems. The method is an augmented Lagrangian method combined with a projected-gradient-type algorithm.

Let us begin with a rough description of the method, consisting of Algorithms \ref{algo1} and \ref{algo2} (page \pageref{algo1}). The second algorithm is a building block of the first one. The augmented Lagrangian method is used to solve the following problem:
\begin{equation} \label{eqPbWithSlack}
\inf_{m \in \text{cl}(\mathcal{R}),\, s \in \R^N} F(m), \quad \text{subject to: } G(m) + s = 0,\ s \geq 0.
\end{equation}
At the end of the while loop of Algorithm \ref{algo1} (line 19), the method provides a probability measure $m_k \in \text{cl}(\mathcal{R})$ satisfying approximately the optimality conditions in variational form for some Lagrange multiplier $\lambda_k$. At this stage, the method has not computed a control process $\bar{u}$ such that $m_k= m_T^{0,Y_0,\bar{u}}$. The ultimate step of the algorithm (line 20) aims at recovering such a control $\bar{u}$ by solving the standard problem \eqref{eqLinearizedPb} with $\phi= DL(m_k,\lambda_k)$. One has to check a posteriori that $\bar{u}$ approximately satisfies the optimality conditions in variational form with associated Lagrange multiplier $\bar{\lambda}= \lambda_k$.

Let us go into the details of the method. The augmented Lagrangian $L_A$ associated with \eqref{eqPbWithSlack} is given by
\begin{equation*}
L_A\colon (m,s,\lambda,c) \in \text{cl}(\mathcal{R}) \times \R^N \times \R^N \times \R_+
\mapsto F(m) + \langle \lambda, G(m) + s \rangle + \frac{c}{2} | G(m) + s |^2.
\end{equation*}
The employed norm in the above definition is the Euclidean norm.
Note that the constraints $s \geq 0$ and $m \in \text{cl}(\mathcal{R})$ are not dualized, since they will be ensured by the projected gradient method.
The mapping $L_A(\cdot,s,\lambda,c)$ is differentiable (in the sense of Assumption \ref{hypDiff}.1) with respect to $m$, with
\begin{align*}
DL_A(m,s,\lambda,c)\hat{m}= \ & DF(m)\hat{m} + \langle \lambda, DG(m) \hat{m} \rangle +  \langle c (G(m)+s), DG(m)\hat{m} \rangle \\
= \ & DL(m,\lambda + c(G(m)+s))\hat{m}, \quad \forall \hat{m} \in \widehat{\mathcal{M}}_p(\R^n).
\end{align*}
A representative of $\hat{m} \mapsto DL_A(m,s,\lambda,c) \hat{m}$ is therefore given by
\begin{align}
DL_A(m,s,\lambda,c,x) = \ & DF(m) + \langle \lambda + c (G(m)+s), DG(m,x) \rangle \notag \\
= \ & DL(m,\lambda + c (G(m)+s),x), \label{eqDerLaug}
\end{align}
up to a constant.
The partial gradient of $L_A$ with respect to $s$ is given by
\begin{equation*}
\nabla_s L_A(m,s,\lambda_c) = \lambda + c(G(m)+s).
\end{equation*}
Let us first focus on Algorithm \ref{algo2}. It aims at solving the following problem:
\begin{equation} \label{eqMinLagAug}
\inf_{m \in \text{cl}(\mathcal{R}),\, s \geq 0} L_A(m,s,\lambda,c),
\end{equation}
for values of the Lagrange multiplier $\lambda$ and the penalty parameter $c>0$ given as input variables. 
The algorithm constructs a sequence $(m_\ell)_{\ell =0,1,...}$ of probability distributions in $\text{cl}(\mathcal{R})$ and a sequence $(s_\ell)_{\ell= 0,1,...}$ of non-negative slack variables in the while loop. At each iteration $\ell$ of the while loop, a kind of line-search (line 10) is performed: the mapping $L_A(\cdot,\cdot,\lambda,c)$ is minimized over the set
\begin{equation} 
\big\{ \big( m_{\ell}(\theta), s_\ell(\theta) \big) \,|\, \theta \in [0,1] \big\},
\end{equation}
where
\begin{equation*}
m_{\ell}(\theta)= (1-\theta)m_\ell + \theta \tilde{m}_\ell, \quad
s_{\ell}(\theta)= \max(s_\ell + \theta \delta s_\ell,0)
\end{equation*}
and where $\delta s_\ell= - \nabla L_A(m_k,s_k,\lambda,c)$. The $\max$ operator in the above expression must be understood coordinatewise, it is nothing but a projection of $s_\ell + \theta \delta s_\ell$ on $\R_{\geq 0}^N$. The probability distribution $\tilde{m}_{\ell}$ is chosen as a solution to
\begin{equation*}
\inf_{m \in \text{cl}(\mathcal{R})} DL_A(m_\ell,s_\ell,\lambda,c)(m-m_{\ell}).
\end{equation*}
The value of the above problem is non-positive.
The measure $\tilde{m}_\ell -m_{\ell}$ can therefore be seen as a descent direction for the variable $m$.
The next iterate of the algorithm is given by
\begin{equation*}
m_{\ell +1}= (1-\theta_{\ell}) m_{\ell} + \theta_{\ell} \tilde{m}_\ell \quad \text{and} \quad
s_{\ell+1}= \max(s_\ell + \theta_{\ell} \delta s_{\ell}, 0),
\end{equation*}
where $\theta_\ell$ minimizes $L_A(m_\ell(\theta),s_\ell(\theta),\lambda,c)$ over $[0,1]$.
Note that Lemma \ref{lemmaConvexity} plays here a crucial role: it guarantees that $m_{\ell+1} \in \text{cl}(\mathcal{R})$. Let us note that Algorithm \ref{algo2} is not, strictly speaking, a projected gradient method, since $m_{\ell +1}$ is not obtained as a projection on $\text{cl}(\mathcal{R})$.

Let us discuss the criterion used in Algorithm \ref{algo2}. It is related to the optimality conditions associated with problem \eqref{eqMinLagAug}: if $(\bar{m},\bar{s})$ is a solution to \eqref{eqMinLagAug}, then, by Theorem \ref{theoOptiCond},
\begin{equation} \label{eqOpCondLa0}
\inf_{m \in \text{cl}(\mathcal{R})} DL_A(\bar{m},\bar{s},\lambda,c)(m-\bar{m})= 0,
\end{equation}
moreover, the optimality of $\bar{s}$ implies that
\begin{equation} \label{eqOpCondLa}
\nabla_s L_A(\bar{m},\bar{s},\lambda,c) \geq 0 \quad \text{and} \quad
\Big[ s_i> 0 \Longrightarrow \nabla_{s_i} L_A(\bar{m},\bar{s},\lambda,c) = 0, \ \forall i=1,...,N \Big].
\end{equation}
Algorithm \ref{algo2} stops when the variable $\varepsilon_{\ell}$ (defined line 6 in the algorithm) is smaller than $\omega$, i.e.\@ when both
\begin{align}
-\inf_{m \in \text{cl}(\mathcal{R})} DL_A(m_{\ell},s_{\ell},\lambda,c)(m-m_{\ell}) \leq \ & \omega, \label{eqOpMLag} \\
| s_{\ell}-\max(s_{\ell} + \delta s_{\ell}, 0) |_\infty \leq \ & \omega. \label{eqOpMLag2}
\end{align}
The inequality \eqref{eqOpMLag} ensures that the optimality condition \eqref{eqOpCondLa0} is approximately satisfied (note that the left-hand side of \eqref{eqOpMLag} is always non-negative).
The inequality \eqref{eqOpMLag2} implies that
\begin{equation*}
s_{\ell}- \max(s_{\ell} + \delta s_{\ell}, 0) \geq -\omega \mathbf{1},
\end{equation*}
thus that
\begin{equation*}
s_{\ell} + \delta s_{\ell} \leq \max(s_{\ell} + \delta s_{\ell},0) \leq s_\ell + \omega \mathbf{1}
\end{equation*}
and finally that
\begin{equation} \label{eqSignCondition}
\nabla_{s} L_A(m_\ell,s_\ell,\lambda,c) = -\delta s_{\ell} \geq - \omega \mathbf{1}.
\end{equation}
Moreover, for all $i=1,...,N$, if $s_{\ell,i} > \omega$,
then $\max(s_{\ell,i} + \delta s_{\ell,i},0) > 0$ (by \eqref{eqOpMLag2}) and therefore,
\begin{equation*}
s_{\ell,i} + \delta s_{\ell,i} = \max (s_{\ell,i} + \delta s_{\ell,i},0)
\end{equation*}
and finally,
\begin{equation} \label{eqCompCond}
\omega \geq | s_{\ell,i}-\max(s_{\ell,i} + \delta s_{\ell,i}, 0) | = |\delta s_{\ell,i}|
= |\nabla_{s_i} L_A(m_\ell,s_\ell,\lambda,c)|.
\end{equation}
Inequalities \eqref{eqSignCondition} and \eqref{eqCompCond} therefore ensure that the optimality condition \eqref{eqOpCondLa} is approximately satisfied.

Let us come back to Algorithm \ref{algo1}. It constructs a sequence $(m_k)_{k= 0,1,...}$ of probability distributions in $\text{cl}(\mathcal{R})$, a sequence of non-negative slack variables $(s_k)_{k= 0,1,...}$ and a sequence of Lagrange multipliers $(\lambda_k)_{k= 0,1,...}$. Two sequences of tolerances are also constructed, $(\eta_k)_{k=0,1,...}$ and $(\omega_k)_{k=0,1,...}$, as well as a sequence of penalty parameters $(c_k)_{k= 0,1,...}$.
At the iteration $k$, the augmented Lagrangian is minimized by using Algorithm \ref{algo2} with $\lambda= \lambda_k$ and the current tolerance $\omega_k$. The triplet $(m_{k+1},s_{k+1},\varepsilon_{k+1})$ is the output of Algorithm \ref{algo2}.
Three cases are then considered.
\begin{itemize}
\item If $| G(m_{k+1}) + s_{k+1} | \leq \eta_k$, then we consider that the penalty term $c_k$ is large enough. The Lagrange multiplier is updated as follows:
\begin{equation*}
\lambda_{k+1}= \nabla_s L_A(m_{k+1},s_{k+1},\lambda_k,c_k)= \lambda_k + c_k(G(m_k)+s_k).
\end{equation*}
This update rule is motivated by \eqref{eqDerLaug}.
\begin{itemize}
\item If moreover $| G(m_{k+1}) + s_{k+1} | \leq \eta_*$ and $\varepsilon_{k+1} \leq \omega_*$, then the algorithm stops and we have that
\begin{align*}
& \inf_{m \in \text{cl}(\mathcal{R})} DL_A(m_{k+1},s_{k+1},\lambda_k,c_k)(m-m_{k+1}) \\
& \qquad \qquad = \inf_{m \in \text{cl}(\mathcal{R})} DL(m_{k+1},\lambda_{k+1})(m-m_{k+1})
\geq  -\omega_*.
\end{align*}
Moreover, by \eqref{eqSignCondition}, $\lambda_{k+1} \geq - \omega_* \mathbf{1}$ and for all $i=1,...,N$, if $G_i(m_{k+1}) > \eta_* + \omega_*$, then
\begin{equation*}
s_{k+1,i} \geq |G_i(m_{k+1})|-|G_i(m_{k+1})+s_{k+1,i}| \geq \omega_*
\end{equation*}
and therefore $|\lambda_{k+1,i}| \leq \omega_*$ by \eqref{eqCompCond}. The probability distribution $m_{k+1}$ satisfies approximately the optimality conditions in variational form, with associated Lagrange multiplier $\lambda_{k+1}$.
\item Otherwise, the penalty term is unchanged and the tolerances $\eta_k$ and $\omega_k$ are tightened (line 12).
\end{itemize}
\item  If $| G(m_{k+1}) + s_{k+1} | > \omega_k$, then the penalty term $c_k$ is regarded as too weak, it is therefore increased. The estimate of the Lagrange multiplier $\lambda$ is unchanged and the tolerances are re-initialized (line 16).
\end{itemize}

\begin{remark} \label{remarkDifficulty}
In practice, the main difficulty in the method is the resolution of the standard problem. It consists of two phases: in a backward phase, the Hamilton-Jacobi-Bellman equation associated with the standard problem must be solved (see subsection \ref{subSecStandardPb}). It provides an optimal control (for the standard problem) in a feedback form. One must then compute the probability distribution $(m_t)_{t \in [0,T]}$ which is associated, in a forward phase.
\end{remark}

\begin{remark} \label{remarkLancelot}
Algorithm 1 is taken from \cite[Algorithm 17.4]{NW06} (see also \cite{CGT92}). Note that the update rules for $\eta_k$ and $\omega_k$ have been modified, in order to avoid too strong variations of the parameters $\eta_k$ and $\omega_k$.
\end{remark}

\begin{algorithm}
Input: $m_{\text{init}} \in \text{cl}(\mathcal{R})$, $s_{\text{init}} \in \R_{\geq 0}^N$, $\lambda_{\text{init}} \in \R^N$, $\eta_*>0$, $\omega_*>0$\;
Set $k=0$, $c_0= 10$, $\omega_0= 1/c_0$, $\eta_0= 1/c_0^{0.1}$\;
Set $m_0= m_{\text{init}}$, $\lambda_0= \lambda_{\text{init}}$, $s_0= s_{\text{init}}$\;
Set $T= \text{true}$\;
\While{$T= \text{\emph{true}}$}{
Compute $(m_{k+1},s_{k+1},\varepsilon_{k+1})$ as an output of Algorithm \ref{algo2} with input $(m_k,s_k,\lambda_k,c_k,\omega_k)$\;
\eIf{$| G(m_{k+1}) + s_{k+1} | \leq \eta_k$}{
Set $\lambda_{k+1}= \lambda_k + c_k(G(m_{k+1})+s_{k+1})$\;
\eIf{$| G(m_{k+1})+s_{k+1} | \leq \eta_*$ \text{and} $\varepsilon_{k+1} \leq \omega_*$}{
Set $T= \text{false}$\;
}
{
Set $c_{k+1}= c_k$, $\eta_{k+1}= \eta_k/10^{0.1}$, $\omega_{k+1}= \omega_k/10$\;
}
}
{
Set $\lambda_{k+1}= \lambda_k$\;
Set $c_{k+1}= 10 c_k$, $\eta_{k+1}= 1/c_{k+1}^{0.1}$, $\omega_{k+1}= 1/c_{k+1}$\;
}
Set $k= k+1$\;
}
Compute a solution $\bar{u}$ to \eqref{eqLinearizedPb} with $\phi(x)= DL(m_k,\lambda_k,x)$\;
Output: $\bar{u}$, $\bar{\lambda}= \lambda_k$.
\caption{Augmented Lagrangian method for solving Problem \eqref{eqPb}}
\label{algo1}
\end{algorithm}

\begin{algorithm}
Input: $m_{\text{init}} \in \text{cl}(\mathcal{R})$, $s_{\text{init}} \in \R_{\geq 0}^N$, $\lambda \in \R^N$, $c>0$, $\omega>0$\;
Set $\ell= 0$, $m_0= m_{\text{init}}$, $s_0= s_{\text{init}}$, $T= \text{true}$\;
\While{$T=\text{\emph{true}}$}{
Compute a solution $\tilde{m}_{\ell}$ to: $\inf_{m \in \text{cl}(\mathcal{R})} DL_A(m_{\ell},s_{\ell},\lambda,c)(m-m_\ell)$\;
Set $\delta s_{\ell}=  -(\lambda + c(G(m_\ell)+ s_\ell))$\;
Set $\varepsilon_{\ell}= \max\big( DL_A(m_{\ell},s_{\ell},\lambda,c)(m_\ell-\tilde{m}_\ell), | s_{\ell} - \max(s_\ell + \delta s_\ell, 0) |_\infty \big)$\;
\eIf{$\varepsilon_{\ell} \leq \omega$}{
Set $T= \text{false}$ \;
}
{
Compute a solution $\theta_{\ell}$ to: $\inf_{\theta \in [0,1]} L_A \big( (1-\theta) m_{\ell} + \theta \tilde{m}_\ell,\max(s_{\ell}+ \theta \delta s_\ell, 0), \lambda,c \big)$\;
Set $m_{\ell+1}= (1-\theta_{\ell}) m_{\ell} + \theta_{\ell} \tilde{m}_\ell$, $s_{\ell +1} = \max(s_{\ell}+ \theta \delta s_\ell, 0)$\;
Set $\ell = \ell+1$\;
}
}
Output: $m_{\ell}$, $s_{\ell}$, $\varepsilon_{\ell}$.
\caption{Projected gradient method for minimizing $L_A(\cdot,\cdot,\lambda,c)$ on $\text{cl}(\mathcal{R}) \times \R_{\geq 0}^N$}
\label{algo2}
\end{algorithm}

\subsection{Convergence analysis}

We investigate in this subsection the termination of Algorithms \ref{algo1} and \ref{algo2}. Our analysis follows the main lines of \cite[Chapters 3 and 17]{NW06}.
Let us mention that we do not tackle in this subsection the issues related to discretization.
In general, termination proofs for line-search methods require that the function to be minimized is differentiable with a Lipschitz-continuous gradient. A similar assumption is therefore considered below.

\begin{assumption} \label{hypLipGrad}
The mappings $F$ and $G$ are differentiable in the sense of Assumption \ref{hypDiff}. Moreover, there exist two constants $K_1>0$ and $K_2 > 0$ such that for all $m_1$, $m_2$, $m_3$, and $m_4 \in \bar{B}_p(R)$,
\begin{equation}
\begin{cases}
\begin{array}{l}
|(DF(m_2)-DF(m_1))(m_4-m_3)| \leq K_1 d_1(m_1,m_2) d_1(m_3,m_4), \\
|(DG(m_2)-DG(m_1))(m_4-m_3)| \leq K_1 d_1(m_1,m_2) d_1(m_3,m_4),
\end{array}
\end{cases}
\label{eqLipGrad1}
\end{equation}
and such that
\begin{align*}
& |DF(m_1)(m_4-m_3)| \leq K_2 d_1(m_3,m_4), \\
& |DG(m_1)(m_4-m_3)| \leq K_2 d_1(m_3,m_4).
\end{align*}
\end{assumption}

\begin{remark}
\begin{enumerate}
\item It can be easily checked that inequality \eqref{eqDirDer2} is a consequence of \eqref{eqLipGrad1}.
\item It can be easily proved that under Assumption \ref{hypLipGrad}, $F$ and $G$ are Lipschitz continuous for the distance $d_1$, with modulus $K_2$.
\end{enumerate}
\end{remark}

Before starting the convergence analysis, we give an example of a mapping satisfying Assumption \ref{hypLipGrad}.

\begin{lemma}
If $F$ and $G$ are of the form $m \in \bar{B}_p(R) \mapsto \Psi(\int \phi \dd m)$ where $\Psi$ is differentiable with a Lipschitz continuous derivative on bounded sets, and where $\phi$ is globally Lipschitz continuous, then Assumption \ref{hypLipGrad} holds.
\end{lemma}

\begin{proof}
Recall the expression of the derivative, given in this situation by \eqref{eqDiffFunctionExp}.
Let $K_a$ be the Lipschitz modulus of $\phi$. For $m \in \bar{B}_p(R)$, by H\"older's inequality,
\begin{align*}
\Big| \int_{\R^n} \phi(x) \dd m(x) \Big|
\leq \ & \int_{\R^n} | \phi(0) | + K_a | x | \dd m(x) \\
\leq \ & | \phi(0) | + K_a \int_{\R^n} | x |^p \dd m(x)\\
\leq \ & | \phi(0) | + K_a R =: K_b.
\end{align*}
Let $K_c$ be the Lipschitz modulus of $D\Psi$ on the ball of centre 0 and radius $K_b$. Let $K_d$ be a bound of $| D\Psi |$ on the same ball. Using the dual representation of the Wasserstein distance $d_1$, we obtain that
\begin{equation*}
\Big| \int_{\R^n} D\Psi \big( {\textstyle\int } \phi \dd m_1 \big) \phi(x) \dd \big( m_3(x) - m_4(x) \big) \Big|
\leq K_a K_d \, d_1(m_3,m_4),
\end{equation*}
since $x \in \R^n \mapsto D\Psi \big( {\textstyle\int } \phi \dd m_1 \big) \phi(x)$ is Lipschitz continuous with modulus $K_aK_d$. Similarly, we also obtain that
\begin{align*}
\Big| \int_{\R^n} \big( D\Psi\big( {\textstyle\int } \phi \dd m_2 \big) - D\Psi\big( {\textstyle\int } \phi \dd m_1 \big) \big) \phi(x) \dd (m_3(x) - m_4(x)) \Big|
\leq \ & K_a^2 K_c d_1(m_1,m_2) d_1(m_3,m_4),
\end{align*}
since
\begin{equation*}
\big| D\Psi\big( {\textstyle\int } \phi \dd m_2 \big) - D\Psi\big( {\textstyle\int } \phi \dd m_1 \big) \big| \leq K_a K_c d_1(m_1,m_2).
\end{equation*}
Thus, Assumption \ref{hypLipGrad} holds with $K_1= K_a^2 K_c$ and $K_2= K_a K_d$.
\end{proof}

The following lemma provides some useful properties dealing with the Lipschitz-continuity of the derivatives of the augmented Lagrangian.

\begin{lemma} \label{lemmaRegLA}
Under Assumption \ref{hypLipGrad}, for all $\lambda \in \R^N$, for all $c > 0$, for all bounded sets $\mathcal{S}$, there exist three constants $K_3$, $K_4$ and $K_5 > 0$ such that for all $m_1$ and $m_2 \in \text{\emph{cl}}(\mathcal{R})$, for all $s_1$ and $s_2 \in \mathcal{S}$,
\begin{align*}
| \nabla_s L_A(m_1,s_1,\lambda,c) | \leq \ & K_3, \\
| \nabla_s L_A(m_2,s_2,\lambda_c)- \nabla_s L_A(m_1,s_1,\lambda,c) | \leq \ & K_4(d_1(m_1,m_2)+ | s_2-s_1|), \\
 | (DL_A(m_2,s_2,\lambda,c)-DL_A(m_1,s_1,\lambda,c))(m_4-m_3) | 
 \leq \ & K_5 (d_1(m_1,m_2)+ | s_2-s_1 |)d_1(m_3,m_4).
\end{align*}
\end{lemma}

\begin{proof}
It is clear that $| \nabla_s L_A(m_1,s_1,\lambda,c) | = | \lambda + c(G(m_1)+s_1) |$ is bounded, since $G$ is Lipschitz continuous and since $\text{cl}(\mathcal{R})$ and $\mathcal{S}$ are bounded. The first inequality follows.

We obtain with the Lipschitz-continuity of $G$ that
\begin{align*}
| \nabla_s L_A(m_2,s_2,\lambda,c) - \nabla_s L_A(m_1,s_1,\lambda_c) |
= \ & c | (G(m_2)-G(m_1)) + (s_2-s_1) | \\
\leq \ & c (K_2 d_1(m_1,m_2)+ | s_2 -s_1 |),
\end{align*}
which proves the second inequality.

For proving the third inequality, we focus on the Lipschitz continuity of the mapping $(s,m) \mapsto (G(m)+s)^\top DG(m)$ (the other terms involved in $DL_A(\cdot,\cdot,\lambda,c)$ can be easily treated).
Let $K$ and $S>0$ be such that for all $m_1 \in \text{cl}(\mathcal{R})$, for all $s \in \mathcal{S}$, $| G(m_1) | \leq K$ and $| s_1 | \leq S$.
We have
\begin{align*}
& \big| \big( (G(m_2)+s_2)^\top DG(m_2)-(G(m_1)+s_1)^\top DG(m_1) \big)(m_4-m_3) \big| \\
& \qquad \qquad \leq \big| (G(m_2)-G(m_1)+s_2-s_1)^\top DG(m_2)(m_4-m_3) \big| \\
& \qquad \qquad \qquad \qquad + \big| (G(m_1)+s_1)^\top (DG(m_2)-DG(m_1))(m_4-m_3) \big| \\
& \qquad \qquad \leq \big( K_2 d_1(m_1,m_2) + | s_2 - s_1 | \big) K_2 d_1(m_3,m_4) \\
& \qquad \qquad \qquad \qquad + (K+S)K_1 d_1(m_1,m_2) d_1(m_3,m_4) \\
& \qquad \qquad \leq \max \big( K_2^2 + (K+S)K_1, K_2 \big) \big( d_1(m_1,m_2) + | s_2-s_1 | \big) d_1(m_3,m_4).
\end{align*}
The third inequality follows.
\end{proof}

\begin{proposition} \label{propInnerLoop}
Under Assumption \ref{hypLipGrad}, Algorithm \ref{algo2} terminates.
\end{proposition}

\begin{proof}
We do a proof by contradiction and therefore assume that the algorithm never terminates. Therefore, it generates a sequence $(m_\ell,s_\ell,\varepsilon_\ell)_{\ell \in \mathbb{N}}$ which is such that
$\varepsilon_{\ell} > \omega$, for all $\ell \in \mathbb{N}$.
One can easily prove that the following set is bounded:
\begin{equation*}
\mathcal{S}:= \{ s \in \R^N \,|\, \exists m \in \text{cl}(\mathcal{R}),\ L_A(m,s,\lambda,c) \leq L_A(m_0,s_0,\lambda,c) \},
\end{equation*}
since $G$ is bounded on $\text{cl}(\mathcal{R})$ and since for a fixed $m \in \text{cl}(\mathcal{R})$, $s \in \R^N \mapsto L_A(m,s,\lambda,c)$ is linear-quadratic, with a dominant term $ c | s |^2$ independent of $m$.
In a similar way, one can prove that $L_A(\cdot,\cdot,\lambda,c)$ is bounded from below.
By construction, the sequence $(L_A(m_\ell,s_\ell,\lambda,c))_{\ell \in \mathbb{N}}$ is decreasing, therefore, for all $\ell \in \mathbb{N}$, $s_\ell \in \mathcal{S}$. Let $K_3$, $K_4$, and $K_5$ be the three constants given by Lemma \ref{lemmaRegLA}, for the set $\mathcal{S}$.

The proof mainly consists in finding an upper estimate of the decay
\begin{equation*}
L_A(m_{\ell+1},s_{\ell +1},\lambda,c)-L_A(m_{\ell},s_{\ell},\lambda,c),
\end{equation*}
at a given iteration $\ell$. This is achieved with estimate \eqref{eqEstimateDecay} below.
Let us introduce some notation, used only in this proof. For $\theta \in [0,1]$, we denote
\begin{equation*}
m(\theta)= (1-\theta) m_{\ell} + \theta \tilde{m}_{\ell} \quad \text{and} \quad
s(\theta)= \max(s_{\ell} + \theta \delta s_{\ell},0).
\end{equation*}
We also omit the arguments $\lambda$ and $c$ of the augmented Lagrangian (since they are fixed).
Let $\theta \in [0,1]$. Observe first that by Lemma \ref{lemmaConvComb},
\begin{equation} \label{eqEstimM}
d_1(m(0),m(\theta)) \leq \theta d_1(m_\ell,\tilde{m}_\ell) \leq \theta D,
\end{equation}
where $D$ is the diameter of $\text{cl}(\mathcal{R})$ (defined by \eqref{eqDiameter}).
Let us estimate $| s(\theta)-s(0) |$. By Lemma \ref{lemmaRegLA}, $| \delta s_\ell | = | \nabla_s L_A(m_\ell,s_\ell,\lambda,c) | \leq K_3$. Since $s_{\ell} \geq 0$ and since the mapping $s \in \R^N \mapsto \max(s,0)$ is Lipschitz-continuous with modulus 1 (it is a projection), we have
\begin{equation} \label{eqEstimS}
| s(\theta)-s(0) |
= | \max(s_\ell + \theta \delta s_{\ell}, 0) - \max(s_\ell, 0) |
\leq | s_\ell + \theta \delta s_\ell - s_\ell |
= \theta | \delta s_\ell |
\leq \theta K_3.
\end{equation}
Now, we split the decay into two terms as follows:
\begin{align*}
& L_A(m(\theta),s(\theta)) - L_A(m(0),s(0)) \\
& \qquad \qquad = \underbrace{L_A(m(\theta),s(\theta))-L_A(m(\theta),s(0))}_{(a)} +
\underbrace{L_A(m(\theta),s(0))-L_A(m(0),s(0))}_{(b)}.
\end{align*}
We split the first term as follows:
\begin{align*}
(a) = \ & \underbrace{\int_0^1 \big\langle \nabla_s L_A \big( m(\theta),s(0)+\xi (s(\theta)-s(0)) \big) - \nabla_s L_A \big( m(0),s(0) \big), s(\theta)-s(0) \big\rangle \dd \xi}_{(a_1)} \\
& \qquad + \underbrace{\langle \nabla_s L_A(m(0),s(0)), s(\theta)-s(0) \rangle}_{(a_2)}.
\end{align*}
Combining Lemma \ref{lemmaRegLA} with estimates \eqref{eqEstimM} and \eqref{eqEstimS}, we obtain that
\begin{equation*}
(a_1) \leq K_4 (D\theta + K_3 \theta) K_3 \theta = K_3 K_4(D+K_3) \theta^2.
\end{equation*}
Since $s(\theta)$ is the orthogonal projection of $s_{\ell}+ \theta \delta s_{\ell}$ on $\R_{\geq 0}^N$ and since $s_\ell= s(0) \in \R_{\geq 0}^N$, we have
\begin{equation*}
\langle s(0) - s(\theta) , s(0) + \theta \delta s_\ell - s(\theta) \rangle \leq 0.
\end{equation*}
Using $\delta s_\ell= - \nabla_s L_A(m_\ell,s_\ell,\lambda,c)$, we deduce that
\begin{equation*}
\theta \langle \nabla_s L_A(m_\ell,s_\ell,\lambda,c), s(\theta)-s_0 \rangle \leq - | s(\theta)-s(0) |^2.
\end{equation*}
It is proved in \cite[Lemma 2.2]{CM87} that
\begin{equation*}
| s(\theta)-s(0) | \geq \theta | s(1)-s(0) |.
\end{equation*}
Combining the last two estimates, we obtain that
\begin{equation*}
(a_2)= \langle \nabla_s L_A(m_\ell,s_\ell,\lambda,c), s(\theta)-s(0) \rangle
\leq -\theta | s(1)-s(0) |^2
= -\theta | s_\ell - \max(s_\ell + \delta s_\ell, 0) |^2.
\end{equation*}
Let us estimate $(b)$. We have
\begin{align*}
(b) = \ & \int_0^\theta \big( DL_A(m(\xi),s(0))-DL_A(m(0),s(0)) \big) (\tilde{m}_\ell - m_\ell ) \dd \xi \\
& \qquad + \theta DL_A(m(0),s(0))(\tilde{m}_\ell-m_\ell) \\
\leq \ & \int_0^\theta K_5 D \xi D \dd \xi + \theta DL_A(m_\ell,s_\ell)(\tilde{m}_\ell-m_\ell) \\
= \ & \frac{1}{2} K_5 D^2 \theta^2 + \theta DL_A(m_\ell,s_\ell)(\tilde{m}_\ell-m_\ell).
\end{align*}
For all $\ell \in \mathbb{N}$, we denote
\begin{equation*}
\tilde{\varepsilon}_\ell
= DL_A(m_\ell,s_\ell,\lambda,c)(m_\ell-\tilde{m}_\ell) + | s_\ell - \max(s_\ell + \delta s_\ell,0) |^2 \geq 0.
\end{equation*}
Combining the three obtained upper estimates of $(a_1)$, $(a_2)$, and $(b)$, we obtain that there exists a constant $K$, independent of $\ell$, such that for all $\theta \in [0,1]$,
\begin{equation} \label{eqEstimateDecay}
L_A(m(\theta),s(\theta))-L_A(m(0),s(0)) \leq \frac{1}{2} K \theta^2 - \theta \tilde{\varepsilon}_\ell.
\end{equation}
For all $\ell \in \mathbb{N}$, we define
\begin{equation*}
\tilde{\theta}_\ell = \min \Big( 1, \frac{\tilde{\varepsilon}_\ell}{K} \Big).
\end{equation*}
If $\tilde{\theta}_\ell = \frac{\tilde{\varepsilon}_\ell}{K}$, then
\begin{equation*}
L_A(m(\tilde{\theta}_\ell),s(\tilde{\theta}_\ell))-L_A(m(0),s(0))
\leq - \frac{1}{2K} \tilde{\varepsilon}_\ell^2.
\end{equation*}
Otherwise, $\tilde{\theta}_{\ell}=1$ and $K \leq \tilde{\varepsilon}_\ell$. Therefore
\begin{equation*}
L_A(m(\tilde{\theta}_\ell),s(\tilde{\theta}_\ell))-L_A(m(0),s(0))
\leq \frac{1}{2} K - \tilde{\varepsilon}_\ell \leq - \frac{1}{2} \tilde{\varepsilon}_\ell.
\end{equation*}
Therefore, for all $\ell \in \mathbb{N}$,
\begin{equation*}
L_A(m(\tilde{\theta}_\ell),s(\tilde{\theta}_\ell))-L_A(m(0),s(0))
\leq -\frac{1}{2} \min \Big( \frac{1}{K} \tilde{\varepsilon}_\ell^2, \tilde{\varepsilon_\ell} \Big)
\end{equation*}
and thus
\begin{align*}
L_A(m_{\ell+1},s_{\ell +1})- L_A(m_\ell,s_\ell)
\leq \ & L_A(m(\tilde{\theta}_\ell),s(\tilde{\theta}_\ell)) - L_A(m(0),s(0)) \\
\leq \ & -\frac{1}{2} \min \Big( \frac{1}{K} \tilde{\varepsilon}_\ell^2, \tilde{\varepsilon_\ell} \Big).
\end{align*}
Recall that the sequence $(L_A(m_\ell,s_\ell))_{\ell \in \mathbb{N}}$ is bounded from below. Let $\bar{L}_A$ be a lower bound. We deduce from the above estimate that for all $q \in \mathbb{N}$,
\begin{equation*}
\sum_{\ell= 0}^q \min \Big( \frac{1}{K} \tilde{\varepsilon}_\ell^2, \tilde{\varepsilon}_\ell \Big)
\leq 2 \big( L_A(m_0,s_0) - L_A(m_{q+1}, s_{q+1}) \big)
\leq 2 \big(L_A(m_0,s_0) - \bar{L}_A \big).
\end{equation*}
The sequence $\big( \min(K\tilde{\varepsilon}_\ell^2, \tilde{\varepsilon}_\ell) \big)_{\ell \in \mathbb{N}}$ is therefore summable and thus converges to 0. It follows that $(\tilde{\varepsilon}_\ell)_{\ell \in \mathbb{N}}$ converges to 0. Since $\tilde{\varepsilon}_\ell$ is the sum of two non-negative terms, they both converge to 0, i.e.
\begin{equation*}
DL_A(m_\ell,s_\ell,\lambda,c)(m_\ell-\tilde{m}_\ell) \underset{\ell \to \infty}{\longrightarrow} 0 \quad \text{and} \quad
| s_\ell - \max(s_\ell + \delta s_\ell,0) | \underset{\ell \to \infty}{\longrightarrow} 0.
\end{equation*}
It follows that $\varepsilon_\ell \underset{\ell \to \infty}{\longrightarrow} 0$, which is a contradiction.
\end{proof}

\begin{proposition} \label{propTerminationAlgo1}
Under Assumption \ref{hypLipGrad}, if Algorithm \ref{algo1} does not terminate, then any limit point $(\bar{m},\bar{s})$ of $(m_k,s_k)_{k \in \mathbb{N}}$ --- there exists at least one --- is such that
\begin{equation} \label{eqSingMult}
\inf_{m \in \text{\emph{cl}}(\mathcal{R})} \langle G(\bar{m})+\bar{s}, DG(\bar{m})(m-\bar{m}) \rangle = 0
\end{equation}
and such that
\begin{equation} \label{eqSingMult2}
G(\bar{m}) + \bar{s} \geq 0 \quad \text{and} \quad
\Big[ \bar{s}_i > 0 \Longrightarrow G_i(\bar{m}) + \bar{s}_i = 0, \quad \forall i= 1,...,N \Big].
\end{equation}
\end{proposition}

Observe that the two conditions satisfied by $(\bar{m},\bar{s})$ are the optimality conditions for the problem
\begin{equation*}
\inf_{m \in \text{cl}(\mathcal{R}),\, s \in \R^N} | G(m)+s |^2, \quad \text{subject to: } s \geq 0.
\end{equation*}

\begin{proof}[Proof of Proposition \ref{propTerminationAlgo1}]

Let us assume that the algorithm does not terminate.
Let us first prove that there are infinitely many indices $k$ such that $| G(m_{k+1}) + s_{k+1} | > \eta_k$. Suppose that it is not the case, then there exists $k_0$ such that for all $k \geq k_0$, $| G(m_{k+1}) + s_{k+1} | \leq \eta_k$. Considering the update formulas for $\eta_k$ and $\omega_k$ used in this situation (line 12), we obtain that $\eta_k \longrightarrow 0$ and $\omega_k \longrightarrow 0$ and thus for some $k \geq k_0$ sufficiently large, $\eta_k \leq \eta_*$ and $\omega_k \leq \omega_*$. The algorithm necessarily terminates when these two inequalities hold, which is a contradiction.

When $| G(m_{k+1}) + s_{k+1} | \leq \eta_k$, $c_k$ is unchanged and when $| G(m_{k+1})+ s_{k+1}) | > \eta_k$ (which occurs infinitely many times), $c_{k+1}= 10 c_k$. Therefore, $c_k \longrightarrow \infty$. We also have that for all $k \in \mathbb{N}$, $c_k \geq 1$. It is easy to prove by induction that for all $k \in \mathbb{N}$, $\eta_k \leq 1/c_{k}^{0.1}$ and that $\omega_k \leq 1/c_{k+1}$. Therefore, $\eta_k \longrightarrow 0$ and $\omega_k \longrightarrow 0$. For $k$ large enough, say for $k \geq k_1$, $\eta_k \leq \eta_*$ and $\omega_k \leq \omega_*$. Therefore, for $k \geq k_1$, $| G(m_{k+1}) + s_{k+1} | > \eta_k$ (otherwise, the algorithm would terminate). It follows that for $k \geq k_1$, the Lagrange multiplier is not updated anymore: $\lambda_k = \lambda_{k_1}$. We denote this constant value of the Lagrange multiplier by $\lambda$, for simplicity.

We now prove that the sequence $(s_k)_{k \in \mathbb{N}}$ is bounded. Let $K= \sup_{m \in \text{cl}(\mathcal{R})} |F(m)| < \infty$. The following inequalities hold true:
\begin{align} \label{eqEstimatesUpperLower}
-K - | \lambda | \cdot | G(m) + s | + \frac{c}{2} | G(m) + s |^2
\leq \ & L_A(m,s,\lambda,c), \\
L_A(m,s,\lambda,c)
\leq \ & K + | \lambda | \cdot | G(m) + s | + \frac{c}{2} | G(m) + s |^2.
\end{align}
The value of the augmented Lagrangian is decreasing along the iterations of Algorithm \ref{algo2}.
Moreover, the pair $(m_{k+1},s_{k+1})$ is obtained as an output of Algorithm \ref{algo2}, with initial value $(m_k,s_k)$. Therefore,
\begin{equation*}
L_A(m_{k+1},s_{k+1},\lambda,c_k) \leq L_A(m_k,s_k,\lambda,c_k), \quad \forall k \geq k_1.
\end{equation*}
Using \eqref{eqEstimatesUpperLower} and denoting $y_k= G(m_k) + s_k$, we obtain that
\begin{equation*}
- K - | \lambda | \cdot | y_{k+1} | + \frac{c_k}{2} | y_{k+1} |^2
\leq K + | \lambda | \cdot | y_k | + \frac{c_k}{2} | y_k |^2.
\end{equation*}
Dividing by $c_k/2$ and adding $| \lambda |^2/c_k^2$ on both sides and factorizing, we obtain that
\begin{equation*}
- \frac{2K}{c_k} + \Big( | y_{k+1} | - \frac{| \lambda |}{c_k} \Big)^2
\leq \frac{2K}{c_k} + \Big( | y_k | + \frac{| \lambda |}{c_k} \Big)^2.
\end{equation*}
Using the inequality $\sqrt{a+b} \leq \sqrt{a} + \sqrt{b}$, we obtain
\begin{equation*}
| y_{k+1} | - \frac{| \lambda |}{c_k}
\leq | y_k | + \frac{| \lambda |}{c_k} + \frac{2 \sqrt{K}}{\sqrt{c_k}}
\end{equation*}
and finally, by induction, for all $q \geq k_1$,
\begin{equation} \label{eqBoundednessY}
| y_q | \leq | y_{k_1} | + \sum_{k=k_1}^{q-1} \Big( \frac{2 | \lambda |}{c_k} + \frac{2 \sqrt{K}}{\sqrt{c_k}} \Big).
\end{equation}
Since the sequence $(c_k)_{k \geq k_1}$ is a geometric sequence of ratio 10, the sequences $(1/c_k)_{k \geq k_1}$ and $(1/\sqrt{c_k})_{k \geq k_1}$ are also geometric with ratios 1/10 and 1/$\sqrt{10}$, respectively. These last two sequences are therefore summable and we deduce from \eqref{eqBoundednessY} that $(y_k)_{k \in \mathbb{N}}$ is bounded. Since $(G(m_k))_{k \in \mathbb{N}}$ is bounded, we finally obtain that $(s_k)_{k \in \mathbb{N}}$ is bounded.

Since $\text{cl}(\mathcal{R})$ is compact and $(s_k)_{k \in \mathbb{N}}$ bounded, the sequence $(m_k,s_k)_{k \in \mathbb{N}}$ possesses at least one accumulation point. Let $(\bar{m},\bar{s}) \in \text{cl}(\mathcal{R}) \times \R_{\geq 0}^N$ be an accumulation point. To simplify, we assume that the whole sequence $(m_k,s_k)_{k \in \mathbb{N}}$ converges to $(\bar{m},\bar{s})$. The arguments which follow can be easily adapted if only a subsequence converge to $(\bar{m},\bar{s})$. Let $m \in \text{cl}(\mathcal{R})$. For all $k \in \mathbb{N}$, we have
\begin{equation*}
DF(m_k) + \langle \lambda + c_{k-1}(G(m_k)+s_k), DG(m_k)(m-m_k) \rangle \geq -\omega_{k-1}.
\end{equation*}
Dividing by $c_{k-1}$ and passing to the limit, we obtain that
\begin{equation*}
\langle G(\bar{m}) + \bar{s}), DG(\bar{m}) (m-\bar{m}) \rangle \geq 0.
\end{equation*}
Minimizing the left-hand side with respect to $m$, we obtain \eqref{eqSingMult}. It remains to prove \eqref{eqSingMult2}. For all $k \geq k_1$, we have
\begin{equation} \label{eqEquationNotCv}
| s_k - \max(s_k - (\lambda + c_{k-1}(G(m_k) + s_k)), 0) |_\infty \leq \omega_{k-1}.
\end{equation}
Dividing by $c_{k-1}$ and passing to the limit, we obtain that
\begin{equation*}
\max( - (G(\bar{m})+\bar{s}), 0) = 0,
\end{equation*}
which proves that $G(\bar{m})+ \bar{s} \geq 0$. Let $i \in \{ 1,...,N \}$ be such that $\bar{s}_i > 0$. For $k$ large enough, $\omega_{k-1} < s_{k,i}$ and therefore, as a consequence of \eqref{eqEquationNotCv},
\begin{equation*}
\max(s_{k,i}- (\lambda_i + c_{k-1}(G_i(m_k) + s_{k,i})), 0) > 0,
\end{equation*}
meaning that
\begin{equation*}
s_{k,i}- (\lambda_i + c_{k-1}(G_i(m_k) + s_{k,i})) > 0.
\end{equation*}
Dividing by $c_{k-1}$ and passing to the limit, we obtain that $G_i(\bar{m}) + \bar{s}_i \leq 0$ and therefore, we obtain that $G_i(\bar{m}) + \bar{s}_i = 0$, since $G(\bar{m}) + \bar{s} \geq 0$, which proves that \eqref{eqSingMult2} holds.
\end{proof}

\subsection{Results}

We present numerical results for two academical problems. The considered SDE is the following for the two of them:
\begin{equation} \label{eqDynForResults}
\dd X_t= u_t \dd t + \dd W_t, \quad X_0= 0, \quad U= [-2,2].
\end{equation}
One can only act on the drift, the volatility is constant and equal to 1. All standard problems are solved by dynamic programming. The corresponding HJB equation is discretized with a semi-Lagrangian scheme (see \cite{CF95}), which consists in approximating the SDE by a controlled Markov chain, defined at times $\{ 0, \delta t,..., T \}$ with $\delta t= 10^{-2}$ and taking values in $\{ -5,-5 + \delta x,..., 5-\delta x, 5 \}$, with $\delta x= 10^{-3}$ and using reflecting boundary conditions. At any mesh point, the minimization of the Hamiltonian is realized by enumeration, for a discretized set of controls $\{ -2, -2 + \delta u,..., 2 \}$ with $\delta u= 10^{-1}$. As mentioned in Remark \ref{remarkDifficulty}, the resolution of standard problems (as the one in Algorithm \ref{algo2}, line 4) is done in two phases. Once an optimal control $u$ has been found by dynamic programming, the corresponding probability distribution is obtained by solving the Chapman-Kolmogorov equation associated with the discretized Markov chain.

The minimization with respect to $\theta$ involved in the computation of a steplength (line 10, Algorithm \ref{algo2}) is done by enumeration. The considered discretized set is $\{ 0, \delta \theta, 2 \delta \theta,...,1 \}$ with $\delta \theta= 10^{-6}$. Note this step of the method is computationally inexpensive (at least for the considered test cases).

The Algorithm \ref{algo1} is initialized with $m_{\text{init}}= \delta_0$ (the Dirac distribution centered at 0), $s_{\text{init}}= 0$, and $\lambda_{\text{init}}= 0$.

Since the SDE \eqref{eqDynForResults} is linear with respect to $u$, the Hamiltonian $H$ is itself linear with respect to $u$ and one can expect that optimal controls only take the boundary values $-2$ and $2$ when the derivative (w.r.t.\@ $x$) of the solution to the HJB equation is positive (resp.\@ negative). The optimal controls obtained below indeed take these values for most of the mesh points. This is why we worked with a rather coarse discretization of $U$.

\paragraph{Test case 1: bounded variance}

For the first test case, we consider the following cost function and constraint:
\begin{equation} \label{eqTestCase1}
F(m)= \int_{\R} x \dd m(x) \quad \text{and} \quad
G(m)= \int_{\R} x^2 \dd m(x) - \Big( \int_{\R} x \dd m(x) \Big)^2 - \alpha,
\end{equation}
where $\alpha= 0.4$. Observe that $F(m_T^{0,Y_0,u})= \mathbb{E}\big[ X_T^{0,Y_0,u} \big]$: the cost function is the expectation of the final state. The mapping $G(m_T^{0,Y_0,u})+ \alpha$ is the variance of $X_T^{0,Y_0,u}$. The cost function $F$ is linear and $G$ is of the form \eqref{eqExample1}. The derivative of $G$ at any $m$, obtained with \eqref{eqDiffFunctionExp}, is a linear-quadratic function (with respect to $x$), given by
\begin{equation*}
DG(m,x)= x^2 - 2 \Big[ \int_{\R} y \dd m(y) \Big] x.
\end{equation*}

Convergence results are shown on Figure \ref{fig:cvTC1}. The tolerances $\eta_*$ and $\omega_*$ are chosen equal, for values ranging from $10^{-3}$ to $10^{-6}$. For each of these tolerances, the values of $G(\bar{u})$ and $\bar{\lambda}$ are provided. The column ``Var.\@ Ineq." contains the following value:
\begin{equation*}
-\inf_{m \in \text{cl}(\mathcal{R})} DL(m_T^{0,Y_0,\bar{u}},\bar{\lambda})(m-m_T^{0,Y_0,\bar{u}}),
\end{equation*}
which somehow indicates to what extent the variational inequality is satisfied.
The column $c$ shows the value of the penalty parameter at the last iteration. The last column shows the total number of standard problems which have to be solved.

It can be first observed that for tolerances below $10^{-4}$, the variational inequality is almost satisfied. The violation of the constraint is small and of the same order as the tolerances. The obtained Lagrange multipliers converge when the tolerance goes to 0. We also observe that the mechanism of Algorithm \ref{algo1} avoids that the penalty term $c$ becomes very high, for small tolerances.

\begin{figure}[htb]
\centering
\begin{tabular}{|c|c|c|c|c|c|}
\hline
Tolerance & $G(\bar{u})$ & $\phantom{|^{a^{a^a}}} \bar{\lambda} \phantom{|^{a^{a^a}}}$ & Var.\@ Ineq. & $c$ & Iterations \\ \hline
1e$-$3 & $3.72\,$e$-3$ & $1.285$ & $1.92\,$e$-5$ & $100$ & 29 \\
1e$-$4 & $7.54\,$e$-4$ & $1.318$ & $1.77\,$e$-15$ & 100 & 39 \\
1e$-$5 & $1.87\,$e$-5$ & $1.324$ & $2.66\,$e$-15$ & 1000 & 60 \\
1e$-$6 & $1.87\,$e$-5$ & $1.324$ & $2.66\,$e$-15$ & 1000 & 60 \\ \hline
\end{tabular}
\caption{Convergence results for the Test Case 1}
\label{fig:cvTC1}
\end{figure}

The optimal control generated by the algorithm is shown on Figure \ref{fig:control_1} (page \pageref{fig:control_1}), the associated probability measure (at any time $t$) is shown on Figure \ref{fig:distribution_1}. The value function associated with the standard problem with cost function $DL(m_T^{0,Y_0,\bar{u}},\bar{\lambda},\cdot)$ is represented on Figure \ref{fig:adjoint1}, we recall that it plays the role of an adjoint equation.

\begin{figure}[p!]
\begin{minipage}[c]{0.49\linewidth}
\begin{center}
\textbf{Test case 1 (problem \eqref{eqTestCase1})}
\begin{subfigure}[c]{\linewidth}
\includegraphics[trim= 3.5cm 8cm 3.5cm 8cm, clip, scale= 0.5]{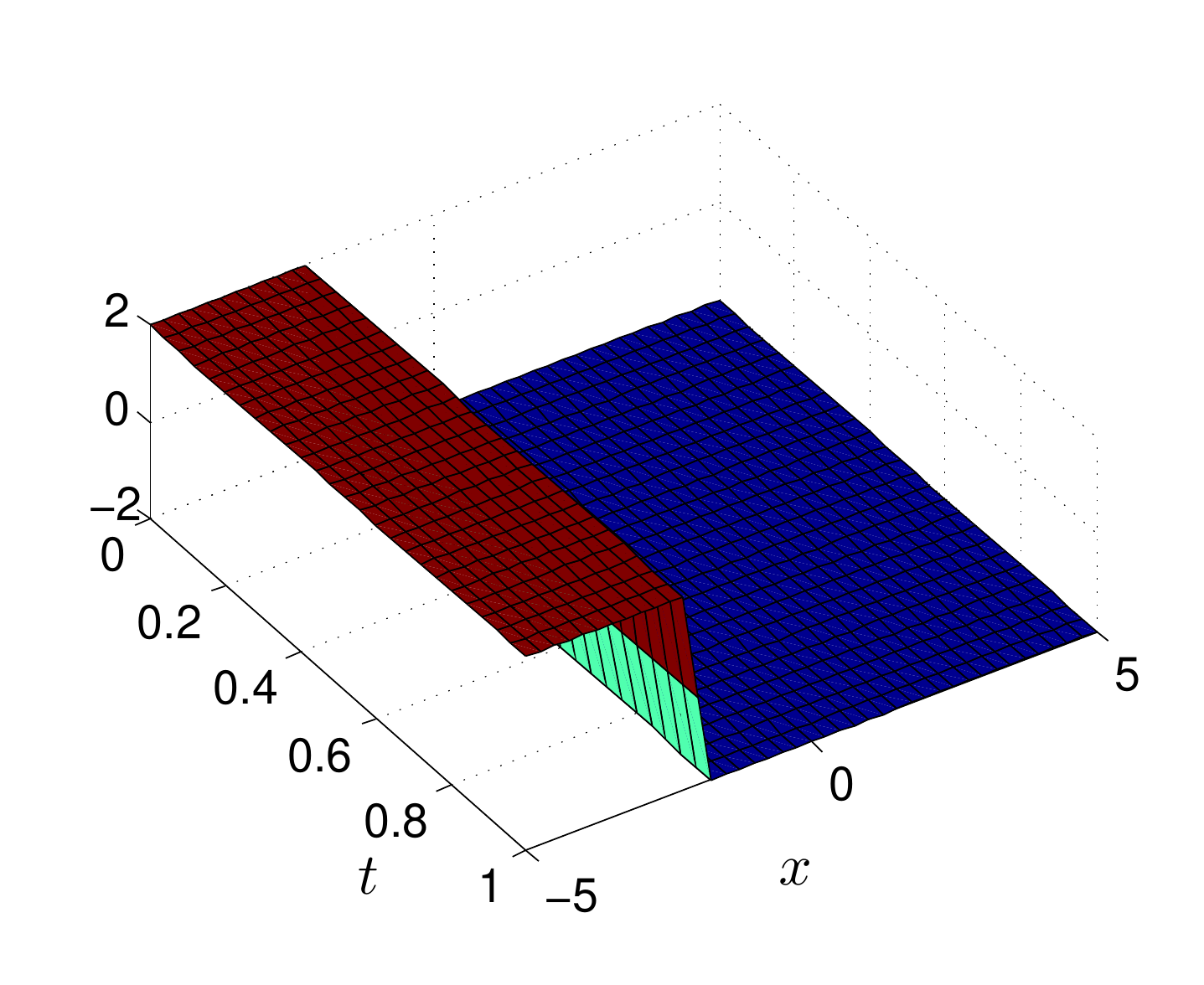}
\caption{Control}
\label{fig:control_1}
\end{subfigure}
\vspace{5mm}
\begin{subfigure}[c]{\linewidth}
\includegraphics[trim= 3.5cm 8cm 3.5cm 8cm, clip, scale= 0.5]{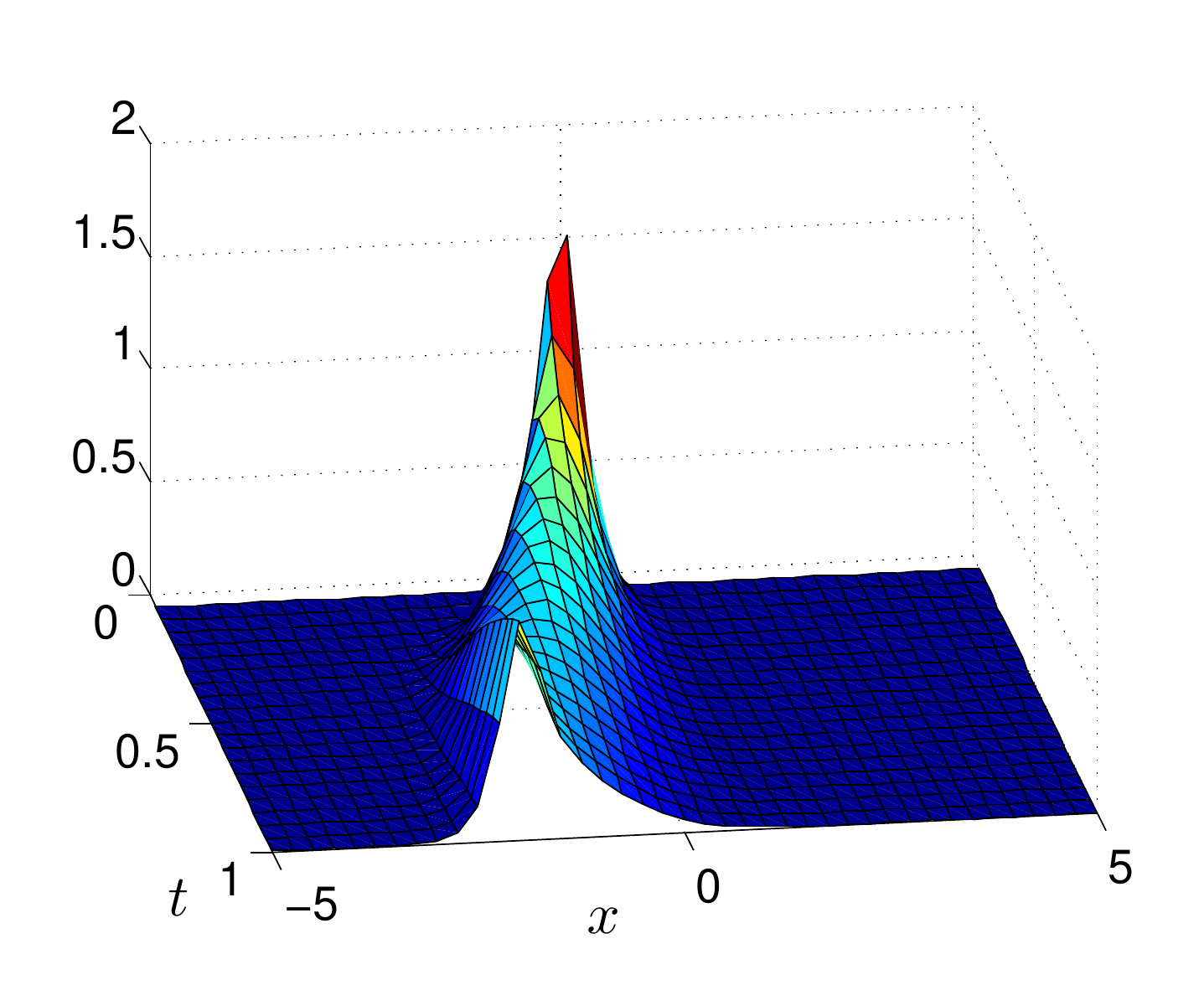}
\caption{Probability distribution}
\label{fig:distribution_1}
\end{subfigure}
\vspace{5mm}
\begin{subfigure}[c]{\linewidth}
\includegraphics[trim= 3.5cm 8cm 3.5cm 8cm, clip, scale= 0.5]{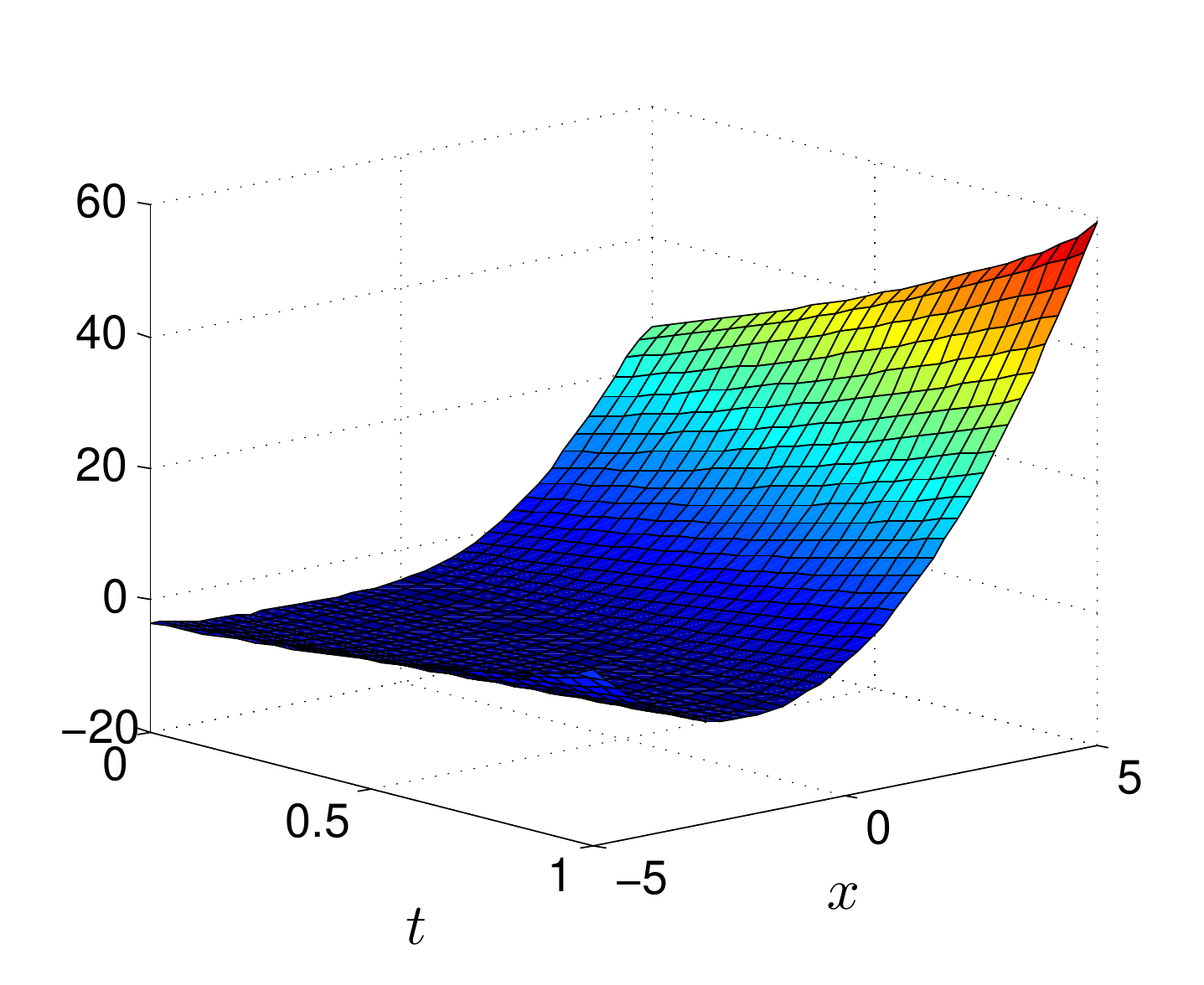}
\caption{Adjoint equation}
\label{fig:adjoint1}
\end{subfigure}
\vspace{10mm}
\end{center}
\end{minipage}
\begin{minipage}[c]{0.49\linewidth}
\begin{center}
\textbf{Test case 2 (problem \eqref{eqTestCase2})}
\begin{subfigure}[c]{\linewidth}
\includegraphics[trim= 3.5cm 8cm 3.5cm 8cm, clip, scale= 0.5]{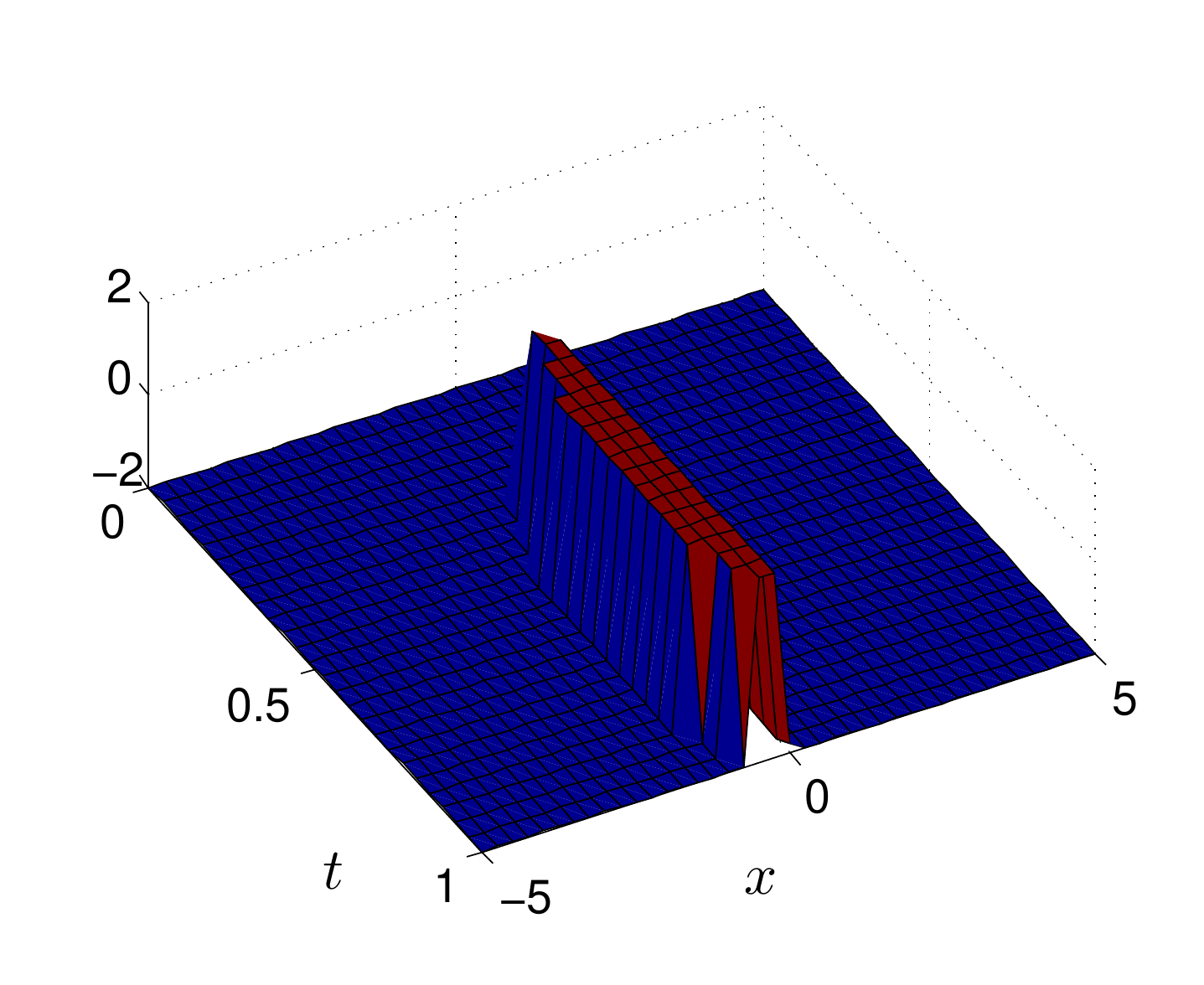}
\caption{Control}
\label{fig:control_2}
\end{subfigure}
\vspace{5mm}
\begin{subfigure}[c]{\linewidth}
\includegraphics[trim= 3.5cm 8cm 3.5cm 8cm, clip, scale= 0.5]{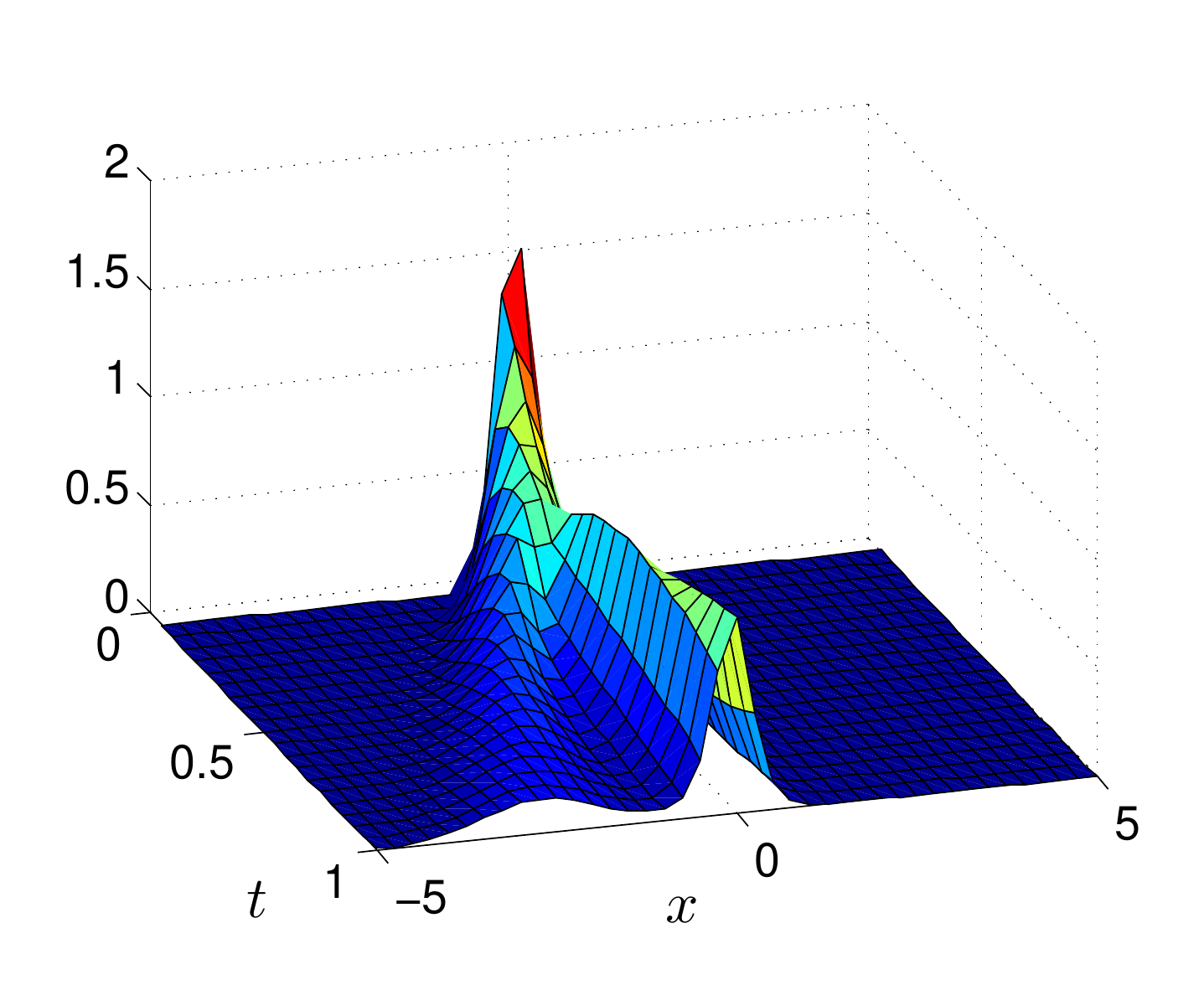}
\caption{Probability distribution}
\label{fig:distribution_2}
\end{subfigure}
\vspace{5mm}
\begin{subfigure}[c]{\linewidth}
\includegraphics[trim= 3.5cm 8cm 3.5cm 8cm, clip, scale= 0.5]{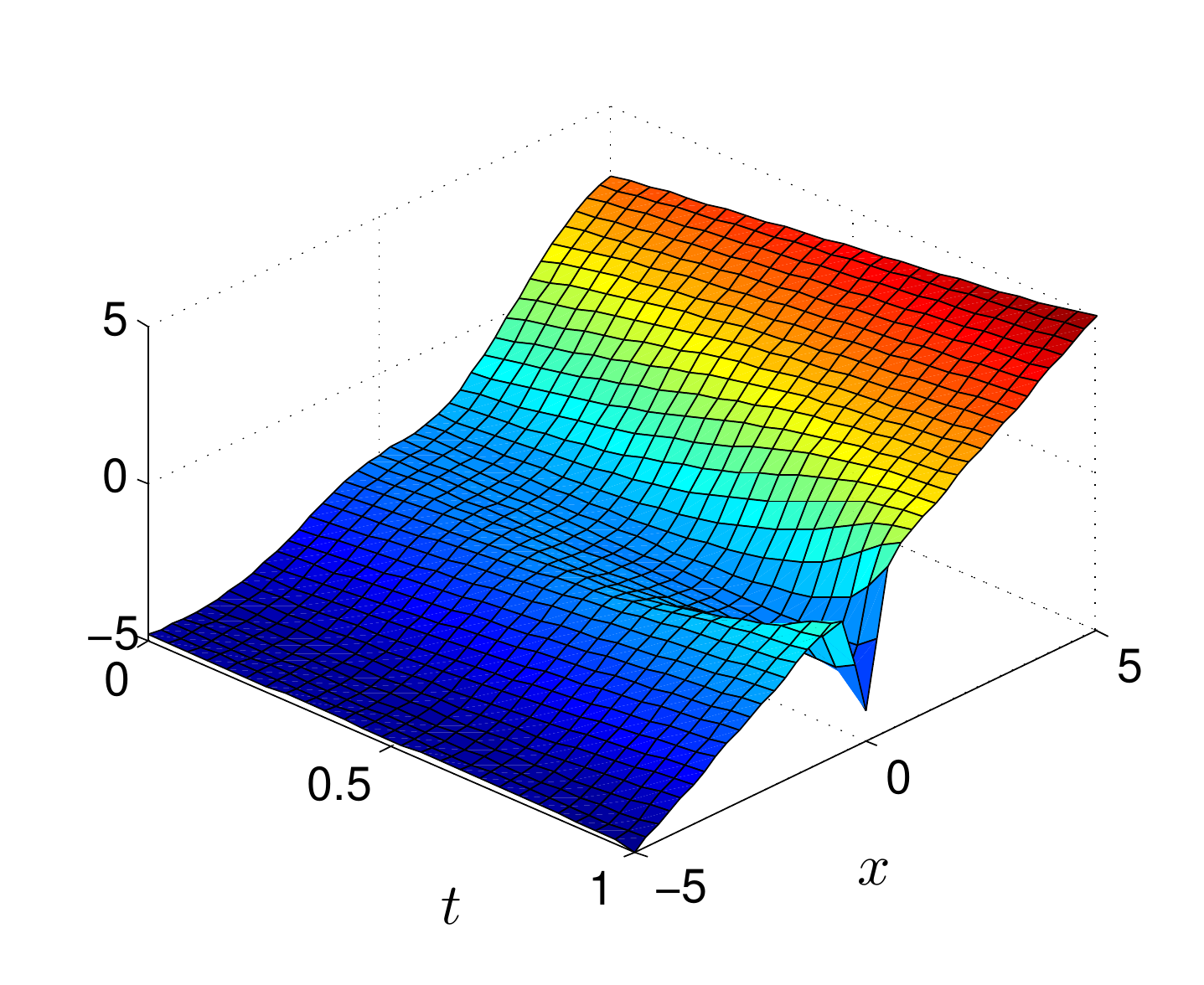}
\caption{Adjoint equation}
\label{fig:adjoint2}
\end{subfigure}
\vspace{10mm}
\end{center}
\end{minipage}
\caption{Numerical results}
\label{fig:numerical_results}
\end{figure}

As expected, the optimal control has a kind of bang-bang structure. It is constant with respect to time, equal to $-2$ for $x \geq -1.6$ and to $2$ for $x \leq -1.6$. If the same problem was solved without constraint, the optimal control would be equal to $-2$, in order to minimize the expectation of the final state. Here, the optimal control must be equal to 2 when $x$ is smaller then $-1.6$ in order to keep the variance sufficiently small and to satisfy the constraint.

\paragraph{Test case 2: expectation constraint}

For this second test case, we consider the following cost function and constraint:
\begin{equation} \label{eqTestCase2}
F(m)= \int_{\R} x \dd m(x) \quad \text{and} \quad
G(m)= -\int_{\R} e^{-x^2} \dd m(x) + \alpha,
\end{equation}
with $\alpha= 0,4$.
Note that both $F$ and $G$ are linear.
Roughly speaking, the constraint $G(m) \leq 0$ ensures that a proportion $\alpha$ of the final probability measure remains around 0.

Convergence results are given in Figure \ref{fig:cvTC2}. The value of $G(\bar{u})$ converges to 0, suggesting that the constraint is active for the undiscretized problem. Convergence of the Lagrange multiplier is observed. The variational inequality is exactly satisfied, since the derivatives of $F$ and $G$ do not depend on $m$. The value of the penalty parameter does not increase much.

\begin{figure}[htb]
\centering
\begin{tabular}{|c|c|c|c|c|c|}
\hline
Tolerance & $G(\bar{u})$ & $\phantom{|^{a^{a^a}}} \bar{\lambda} \phantom{|^{a^{a^a}}}$ & Var.\@ Ineq. & $c$ & Iterations \\ \hline
1e$-$3 & $-1.93\,$e$-2$ & $4.119$ & $0$ & $100$ & $37$ \\
1e$-$4 & $\phantom{-}1.19\,$e$-3$ & $4.024$ & 0 & 100 & 53 \\
1e$-$5 & $-8.22\,$e$-5$ & $4.026$ & 0 & 100 & 64 \\
1e$-$6 & $-8.22\,$e$-5$ & $4.026$ & 0 & 100 & 64 \\ \hline
\end{tabular}
\caption{Convergence results for the Test Case 2}
\label{fig:cvTC2}
\end{figure}

The optimal control, the probability distribution (at any time) and the adjoint are provided in Figures \ref{fig:control_2}, \ref{fig:distribution_2}, and \ref{fig:adjoint2} (page \pageref{fig:control_1}). As can be observed, the optimal control only takes the boundary values. The value $2$ is taken in a small region around $x=0$, after $t \approx 0.4$, which guarantees that a sufficiently large proportion of the distribution remains located around 0, as can be seen on the graph of the probability distribution.

\section{Conclusion}

We have proved optimality conditions for a class of constrained non-linear stochastic optimal control problems, using an appropriate concept of differentiability for the cost function and the constraints. The convexity of the closure of the reachable set of probability measures plays an essential role in the proof of these results. An augmented Lagrangian method, based on the convexity property and the optimality conditions has been proposed, demonstrating the relevance of these properties.
Good convergence results have been obtained for examples with a one-dimensional state variable.
Future work will focus on the extension of these results to more general problems, for example, for cost functions containing an integral cost depending on the current probability distribution.

\paragraph{Acknowledgements}

This study was partly supported by the ERC advanced grant 668998 (OCLOC) under the EU's
H2020 research program.

\appendix

\section{Elements on optimal transportation}

\paragraph{Wasserstein distance}

Let us recall the definition of the Wasserstein distance, denoted by $d_1$ in this article. For all $m_1$ and $m_2$ in $\mathcal{P}_1(\R^n)$,
\begin{equation} \label{eqDefWasserstein}
d_1(m_1,m_2)= \inf_{\pi \in \Pi(m_1,m_2)} \int_{\R^n \times \R^n} |y-x| \dd \pi(x,y),
\end{equation}
the set $\Pi(m_1,m_2)$ being the set of transportation mappings from $m_1$ to $m_2$ defined as:
\begin{equation*}
\Bigg\{ \pi \in \mathcal{P}(\R^{2n}) \,|\, \Big\{ \begin{array}{l} \pi(A\times \R^n)= m_1(A),\\ \pi(\R^n \times A)= m_2(A),\end{array} \, \text{for all measurable $A \subset \R^n$} \Bigg\}.
\end{equation*}

\begin{lemma} \label{lemmaConvComb}
For all $m_1$ and $m_2 \in \mathcal{P}_1(\R^n)$, for all $\theta \in [0,1]$,
\begin{equation} \label{eqConvComb}
d_1((1-\theta)m_1 + \theta m_2) \leq \theta d_1(m_1,m_2).
\end{equation}
\end{lemma}

\begin{proof}
Let $\phi \in \text{1-Lip}(\R^n)$. Then,
\begin{align*}
\int_{\R^n} \phi \dd \big( ((1-\theta)m_1 + \theta m_2)-m_1 \big)
= \theta \int_{\R^n} \phi \dd (m_2-m_1)
\leq \theta d_1(m_1,m_2).
\end{align*}
The last inequality follows from the dual representation \eqref{eqDualWasserstein}. Maximizing the left-hand side with respect to $\phi \in \text{1-Lip}(\R^n)$, we obtain inequality \eqref{eqConvComb}.
\end{proof}

\paragraph{A compactness property}

\begin{lemma} \label{lemmaCompactnessProperty}
For all $p>1$ and $R\geq 0$, the subset $\bar{B}_p(R)$ of $\mathcal{P}_1(\R^n)$ (defined in \eqref{eqDefBpR}) is compact for the $d_1$-distance.
\end{lemma}

\begin{proof}
We first prove that  $\bar{B}_p(R)$ is compact for the weak topology of $\mathcal{P}(\R^n)$. For all $r \geq 0$ and for all $m \in \bar{B}_p(R)$,
\begin{equation} \label{eqIneqDontJaiOublieLeNom}
R \geq \int_{\bar{B}_r^{\text{c}}} |x|^p \dd m(x) \geq r^p \int_{\bar{B}_r^{\text{c}}} 1 \dd m(x),
\end{equation}
and thus, $m(\bar{B}_r^{\text{c}}) \leq R/r^p \rightarrow 0$, meaning that $\bar{B}_p(R)$ is tight. By Prokhorov's theorem \cite[Page 43]{Vil09}, $\bar{B}_p(R)$ is therefore precompact for the weak-topology. Now, let $(m_k)_{k \in \mathbb{N}}$ be a sequence in $\bar{B}_p(R)$ weakly converging to $\bar{m} \in \mathcal{P}(\R^n)$. For all $r \geq 0$, the function $\min(|x|^p,r)$ is continuous and bounded, thus:
\begin{equation*}
\int_{\R^n} \min(|x|^p,r) \dd \bar{m}(x)= \lim_{k \to \infty} \int_{\R^n} \min(|x|^p,r) \dd m_k(x) \leq R.
\end{equation*}
We obtain, using the monotone convergence theorem:
\begin{equation*}
\int_{\R^n} |x|^p \dd \bar{m}(x)=
\lim_{r \to \infty} \int_{\R^n} \min(|x|^p, r) \dd \bar{m}(x) \leq R,
\end{equation*}
thus $\bar{m} \in \bar{B}_p(R)$. Therefore, $\bar{B}_p(R)$ is weakly closed, and thus weakly compact.

Finally, we need to prove that any weakly converging sequence $(m_k)_{k \in \mathbb{N}}$ in $\bar{B}_p(R)$ to some $\bar{m} \in \bar{B}_p(R)$ also converges for the $d_1$-distance. By \cite[Definition 6.8/Theorem 6.9]{Vil09}, it suffices to prove that
\begin{equation} \label{eqToBeProvedForCompacity}
\int_{\R^n} |x| \dd m_k(x) \underset{k \to \infty}{\longrightarrow} \int_{\R^n} |x| \dd \bar{m}(x).
\end{equation}
Observe that for all $r\geq 0$, for all $m \in \bar{B}_p(R)$, similarly to \eqref{eqIneqDontJaiOublieLeNom}, we find that
\begin{equation*}
\Big|\int_{\R^n} |x|-\min(|x|,r) \dd m(x)\Big|
\leq \int_{\bar{B}^{\text{c}}_r} |x| \dd m(x) \leq \frac{R}{r^{p-1}}.
\end{equation*}
Therefore, for all $r \geq 0$,
\begin{equation*}
\limsup_{k \to \infty} \Big|\int_{\R^n} |x| \dd (m_k(x)-\bar{m}(x)) \Big|
\leq \underbrace{\Big( \lim_{k \to \infty} \int_{\R^n} \min( |x|,r ) \dd (m_k(x) - \bar{m}(x)) \Big)}_{=0} + \frac{2R}{r^{p-1}}.
\end{equation*}
We obtain \eqref{eqToBeProvedForCompacity}, making $r$ tend to $+\infty$.
 \end{proof}

\paragraph{A continuity property}

We prove in the following lemma the continuity of linear mappings for the $d_1$-distance on $\bar{B}_p(R)$ under a growth condition.

\begin{lemma} \label{lemmaContinuityDominatedCost}
Let $p>1$, $\phi:\R^n \rightarrow \R$ be dominated by $|x|^p$ (in the sense of \eqref{eqContPropR}).
Then, for all $R \geq 0$, the following mapping: $m \in \bar{B}_p(R) \mapsto \int_{\R^n} \phi(x) \dd m(x)$
is continuous for the $d_1$-distance.
\end{lemma}

\begin{proof}
Let $(m_k)_{k \in \mathbb{N}}$ be a sequence in $\bar{B}_p(R)$ converging to $\bar{m} \in \bar{B}_p(R)$ for the $d_1$-distance.
Let $\varepsilon>0$ and let $r$ be such that \eqref{eqContPropR} holds.
We define $\hat{\phi}: x \in \R^n \mapsto \hat{\phi}(x)= \phi(P(x))$, where $P$ is the orthogonal projection on $\bar{B}_r$.
For all $x \in \R^n$, if $x \in \bar{B}_r$, then $\hat{\phi}(x)= \phi(x)$ and if $x \in \bar{B}_r^\text{c}$, then $|\hat{\phi}(x)|= |\phi(rx/|x|)| \leq \varepsilon r^{p} \leq \varepsilon |x|^p$.
Thus, for all $m \in \bar{B}_p(R)$,
\begin{equation*}
\Big| \int_{\R^n} \hat{\phi}-\phi \dd m \Big|
\leq \Big| \int_{B_r} \hat{\phi}-\phi \dd m \Big|
+ \int_{B^\text{c}_r} |\phi| \dd m + \int_{B^\text{c}_r} |\hat{\phi}| \dd m
\leq  2\varepsilon R.
\end{equation*}
By \cite[Definition 6.8/Theorem 6.9]{Vil09}, the convergence for the $d_1$-distance implies the weak convergence, thus, since $\hat{\phi}$ is continuous and bounded, we obtain:
\begin{align*}
\limsup_{k \to \infty} \Big|\int_{\R^n} \!\phi \dd(m_k - \bar{m}) \Big|
\leq \ & \underbrace{\lim_{k \to \infty} \Big| \int_{\R^n} \!\hat{\phi} \dd(m_k - \bar{m}) \Big|}_{= 0}
 + \limsup_{k \to \infty}  \Big| \int_{\R^n} \!(\hat{\phi}-\phi) \dd(m_k - \bar{m}) \Big| \\
\leq \ & 4\varepsilon R.
\end{align*}
The result follows when $\varepsilon$ tends to 0.
\end{proof}

\end{document}